\documentclass[a4paper,9pt]{article}
\usepackage[utf8]{inputenc}
\usepackage[T1]{fontenc}
\usepackage{lmodern}
\usepackage{stmaryrd}
\usepackage{wrapfig}
\usepackage[ocgcolorlinks,colorlinks=true,linkcolor=black,urlcolor=blue]{hyperref}
\usepackage{latexsym}
\usepackage{amssymb,amsmath,amsthm}
\usepackage{graphicx}
\usepackage{fullpage}
\usepackage[boxed]{algorithm2e}
\usepackage{color}
\usepackage[usenames,dvipsnames]{xcolor}
\usepackage[normalem]{ulem} 
\usepackage{microtype,mathtools}
\usepackage[noabbrev,capitalise]{cleveref}

\newcommand{\R}{\mathbf{R}}

\newcommand{\s}{\mathbf{S}}
\newcommand{\N}{\mathbf{N}}

\newcommand{\A}{\ensuremath{\mathcal{A}}}

\newcommand{\abs}[1]{\left\lvert#1\right\rvert}
\newcommand{\norm}[1]{\left\lVert#1\right\rVert}
\newcommand{\diff}{\mathop{}\mathopen{}\mathrm{d}}
\newcommand{\sst}[2]{\left\{\;#1 \;\middle|\; #2 \;\right\}}

\newcommand{\sign}{\textrm{sign}}

\newcommand{\y}{Y}

\newtheorem{theorem}{Theorem}
\newtheorem{proposition}[theorem]{Proposition}
\newtheorem{coro}[theorem]{Corollary}
\newtheorem{conjecture}[theorem]{Conjecture}
\newtheorem{lemma}{Lemma}[section]

\theoremstyle{definition}
\newtheorem{example}{Example}[section]

\usepackage{tikz}
\usetikzlibrary{backgrounds}

\newcommand{\vc}[1]{\ensuremath{\vcenter{\hbox{#1}}}}
\newcommand{\flagFiveOne}[1][\flagheight]{\hspace{0.2cm}\vc{\includegraphics[page=1,height=#1]{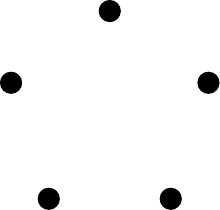}} \hspace{0.2cm}}
\newcommand{\flagFiveTwo}[1][\flagheight]{\hspace{0.2cm}\vc{\includegraphics[page=2,height=#1]{flag0.pdf}} \hspace{0.2cm}}
\newcommand{\flagFiveThree}[1][\flagheight]{\hspace{0.2cm}\vc{\includegraphics[page=3,height=#1]{flag0.pdf}}\hspace{0.2cm} }
\newcommand{\flagFourOne}[1][\flagheight]{\hspace{0.2cm}\vc{\includegraphics[page=4,height=#1]{flag0.pdf}\hspace{0.2cm}} }
\newcommand{\flagFourTwo}[1][\flagheight]{\hspace{0.2cm}\vc{\includegraphics[page=5,height=#1]{flag0.pdf}}\hspace{0.2cm} }
\newcommand{\flagThree}[1][\flagheight]{\hspace{0.2cm}\vc{\includegraphics[page=6,height=#1]{flag0.pdf}}\hspace{0.2cm} }
\newcommand{\flagThreeLab}[1][\flagheight]{\hspace{0.2cm}\vc{\includegraphics[page=7,height=#1]{flag0.pdf}}\hspace{0.2cm} }
\newcommand{\flagFourLabOne}[1][\flagheight]{\hspace{0.2cm}\vc{\includegraphics[page=8,height=#1]{flag0.pdf}}\hspace{0.2cm} }

\newcommand{\flagFourLabFive}[1][\flagheight]{\hspace{0.2cm}\vc{\includegraphics[page=12,height=#1]{flag0.pdf}}\hspace{0.2cm} }

\newcommand{\typeFourOne}[1][\flagheight]{\vc{\includegraphics[page=15,height=#1]{flag0.pdf}} }
\newcommand{\typeFourTwo}[1][\flagheight]{\vc{\includegraphics[page=16,height=#1]{flag0.pdf}} }
\newcommand{\flagThreeLabTwo}[1][\flagheight]{\hspace{0.2cm}\vc{\includegraphics[page=17,height=#1]{flag0.pdf}} \hspace{0.2cm}}
\newcommand{\flagFourLabEight}[1][\flagheight]{\hspace{0.2cm}\vc{\includegraphics[page=18,height=#1]{flag0.pdf}} \hspace{0.2cm}}
\newcommand{\flagFourLabNine}[1][\flagheight]{\hspace{0.2cm}\vc{\includegraphics[page=19,height=#1]{flag0.pdf}} \hspace{0.2cm}}
\newcommand{\flagFourLabTen}[1][\flagheight]{\hspace{0.2cm}\vc{\includegraphics[page=20,height=#1]{flag0.pdf}} \hspace{0.2cm}}
\newcommand{\flagThreeLabThree}[1][\flagheight]{\hspace{0.2cm}\vc{\includegraphics[page=21,height=#1]{flag0.pdf}} \hspace{0.2cm}}
\newcommand{\flagFourLabEleven}[1][\flagheight]{\hspace{0.2cm}\vc{\includegraphics[page=22,height=#1]{flag0.pdf}} \hspace{0.2cm}}
\newcommand{\flagFourLabTwelvea}[1][\flagheight]{\hspace{0.2cm}\vc{\includegraphics[page=23,height=#1]{flag0.pdf}} \hspace{0.2cm}}
\newcommand{\flagFiveLabOne}[1][\flagheight]{\hspace{0.2cm}\vc{\includegraphics[page=24,height=#1]{flag0.pdf}} \hspace{0.2cm}}
\newcommand{\flagFiveLabTwo}[1][\flagheight]{\hspace{0.2cm}\vc{\includegraphics[page=25,height=#1]{flag0.pdf}} \hspace{0.2cm}}
\newcommand{\flagFourLabThirteen}[1][\flagheight]{\hspace{0.2cm}\vc{\includegraphics[page=26,height=#1]{flag0.pdf}} \hspace{0.2cm}}
\newcommand{\flagFourLabFourteen}[1][\flagheight]{\hspace{0.2cm}\vc{\includegraphics[page=27,height=#1]{flag0.pdf}} \hspace{0.2cm}}
\newcommand{\flagFourLabFifteen}[1][\flagheight]{\hspace{0.2cm}\vc{\includegraphics[page=28,height=#1]{flag0.pdf}} \hspace{0.2cm}}
\newcommand{\flagFourLabSixteen}[1][\flagheight]{\hspace{0.2cm}\vc{\includegraphics[page=29,height=#1]{flag0.pdf}} \hspace{0.2cm}}
\newcommand{\flagFourLabTwelveb}[1][\flagheight]{\hspace{0.2cm}\vc{\includegraphics[page=30,height=#1]{flag0.pdf}} \hspace{0.2cm}}

\newcommand\ot{\ensuremath{\omega}}
\newcommand\sot{\ensuremath{\mathcal{O}}}
\newcommand\schi{\ensuremath{\mathcal{X}}}

\renewcommand{\leq}{\leqslant}
\renewcommand{\geq}{\geqslant}
\renewcommand{\le}{\leqslant}
\renewcommand{\ge}{\geqslant}

\newcommand{\pth}[1]{\left(#1 \right)}

\DeclareMathOperator{\cll}{cl_{\|\|_2}}
\DeclareMathOperator{\bll}{\partial_{\|\|_2}}
\DeclareMathOperator{\dist}{d}
\DeclareMathOperator{\Prob}{\mathbf{P}}
\DeclareMathOperator{\Ex}{\mathbf{E}}
\DeclareMathOperator{\Var}{\mathbf{Var}}
\DeclareMathOperator{\supp}{supp}
\DeclareMathOperator{\conv}{conv}
\newcommand\Kern{\ensuremath{\mathfrak{K}}}

\DeclareMathOperator{\Circ}{\mathfrak{C}}
\DeclareMathOperator{\Hom}{Hom}

\DeclareMathOperator{\vect}{vect}

\newcommand\labelset{\ensuremath{\mathcal{Z}}}
\newcommand{\unlab}[2]{\left\llbracket #1\right\rrbracket_{#2}}
\newcommand{\bigo}[1]{O\mathopen{}\left(#1\right)}
\newcommand{\littleo}[1]{o\mathopen{}\left(#1\right)}

\title{Limits of Order Types\thanks{This work was partially supported by ANR blanc PRESAGE (ANR–11-BS02–003).}}

\author{Xavier Goaoc\thanks{%
Universit\'e de Lorraine, CNRS, INRIA, LORIA, F-54000, Nancy, France.
Email: \texttt{xavier.goaoc@loria.fr}. This author was partially supported by Institut Universitaire de France.}
\and Alfredo Hubard\thanks{Universit\'e Paris Est, LIGM (UMR 8049), CNRS, ENPC, ESIEE, UPEM, F-77454, Marne-la-Vall\'ee, France.
Email: \texttt{alfredo.hubard@u-pem.fr}.}
\and R\'emi de Joannis de Verclos\thanks{%
    Radboud University Nijmegen, Netherlands. Email: \texttt{r.deverclos@math.ru.nl}.}
\and Jean-S\'ebastien Sereni\thanks{%
        Centre National de la Rercherche Scientifique, ICube (CSTB), Strasbourg, France.
        Email: \texttt{sereni@kam.mff.cuni.cz}.}
\and Jan Volec\thanks{%
Department of Mathematics, Emory University, Atlanta, USA\@. Email: \texttt{jan@ucw.cz}.
This project has received funding from the European Union’s Horizon~2020
research and  innovation programme under the Marie Skłodowska-Curie grant
agreement No.~800607.  Previous affiliation: Department of Mathematics and
Statistics, McGill University, Montreal, Canada, where this author was
supported by CRM-ISM fellowship.}}

\date{\today}

\begin{document}
\maketitle

\begin{abstract}
  We apply ideas from the theory of limits of dense combinatorial structures to
    study order types, which are combinatorial encodings of finite point sets.
    Using flag algebras we obtain new numerical results on the Erd\H{o}s
    problem of finding the minimal density of 5- or 6-tuples in convex
    position in an arbitrary point set, and also an inequality expressing the
    difficulty of sampling order types uniformly. Next we establish results on
    the analytic representation of limits of order types by planar measures. Our
    main result is a rigidity theorem: we show that if sampling two measures
    induce the same probability distribution on order types, then these
    measures are projectively equivalent provided the support of at least one
    of them has non-empty interior. We also show that some condition on the
    Hausdorff dimension of the support is necessary to obtain projective
    rigidity and we construct limits of order types that cannot be represented by
    a planar measure. Returning to combinatorial
    geometry we relate the regularity of this analytic representation to the aforementioned problem of Erdős
    on the density of $k$-tuples in convex position, for large~$k$. 
\end{abstract}


\section{Introduction}

The theory of \emph{dense graph limits}, developed over the last
decade by Borgs, Chayes, Lov\'asz, Razborov, S\'os, Szegedy,
Vesztergombi and others, studies sequences of large graphs using a
combination of equivalent formalisms: algebraic (as positive
homomorphisms from certain graph algebras into~$\R$), analytic (as
measurable, symmetric functions from~${[0,1]}^2$ to~$[0,1]$ called
\emph{graphons}) and discrete probabilistic (as families~$\{D_n\}$ of
probability distributions over $n$-vertex graphs satisfying certain relations).
These viewpoints are complementary: while the algebraic formalism allows
effective computations \emph{via} semi-definite methods~\cite{Razborov:2007},
the analytic viewpoint offers powerful methods (norm equivalence,
completeness) to treat in a unified setting a diversity of graph problems such
as pseudorandom graphs or property testing~\cite{Lov12}.

In this article, we combine ideas from dense graph
limits with order types, which are combinatorial structures arising in
geometry. The order type of a point set encodes the respective
positions of its elements, and suffices to determine many of its
properties, for instance its convex hull, its triangulations, or which
graphs admit crossing-free straight line drawings with vertices
supported on that point set. Order types have received continued
attention in discrete and computational geometry since the 1980s and
are known to be rather intricate objects, difficult to
axiomatise~\cite{Stretchability}.  While order types are well defined in a variety of
contexts (arbitrary dimensions, abstractly \emph{via} the theory of oriented matroids)
all point sets considered in this work are finite
subsets of the Euclidean plane with no aligned triple, unless otherwise specified.

\subsection{Order types and their limits}

Let us first define our main objects of study. 

\bigskip
\noindent
\begin{minipage}{12cm}
\paragraph{Order types.}
Define the \emph{orientation} of a triangle~$pqr$ in the plane to be
\emph{clockwise} (CW) if~$r$ lies to the right of the line~$pq$
oriented from~$p$ to~$q$ and \emph{counter-clockwise}~(CCW) if~$r$
lies to the left of that oriented line. (So the orientation of~$qpr$
is different from that of~$pqr$.) We say that two planar point sets~$P$ and~$Q$ \emph{have the same order type} if there exists a
bijection~$f\colon P \to Q$ that preserves orientations: for any triple of
pairwise distinct points~$p,q,r \in P$ the triangles~$pqr$ and~$f(p)f(q)f(r)$ have the same orientation.
The relation of having
the same order type is easily checked to be an equivalence relation; \hfill the
equivalence class, for this relation, of 
a finite  point set~$P$ is
called the \emph{order type}
\end{minipage}
\hfill
\begin{minipage}{3.5cm}
\includegraphics[keepaspectratio, width=3.4cm]{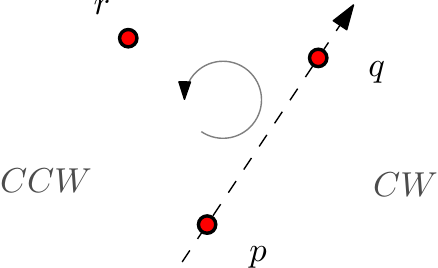}
\end{minipage}

\smallskip
\noindent
of~$P$.
A point set~$P$ with order type~$\ot$ is
called a \emph{realization} of~$\ot$.

\medskip

When convenient, we extend to order types any notion that can be
defined on a set of points and does not depend on a particular choice
of realization. For instance we define the \emph{size} of an order
type~$\ot$ to be the cardinality~$|\ot|$ of any of its realization. We
adopt the convention that there is exactly one order type of each of
the sizes~$0$, $1$ and~$2$. We let~$\sot$ be the set of order types
and~$\sot_n$ the set of order types of size~$n$.

\paragraph{Convergent sequences and limits of order types.}

We define the \emph{density}~$p(\ot,\ot')$ of an order type~$\ot$ in
another order type~$\ot'$ as the probability that~$|\ot|$ random
points chosen uniformly from a point set realizing~$\ot'$ have
order type~$\ot$.  (Observe that this probability depends solely on
the order types and not on the choice of realization.) We say that a
sequence~${(\ot_n)}_{n \in \N}$ of order types \emph{converges} if the
size~$|\ot_n|$ goes to infinity as~$n$ goes to infinity, and for
any fixed order type~$\ot$ the sequence~${(p(\ot,\ot_n))}_n$ of densities
converges. The \emph{limit} of a convergent sequence of
order types~${(\ot_n)}_{n \in \N}$ is the map
\[ \left\{\begin{array}{rcl}\sot & \to & [0,1]\\ \ot & \mapsto & \lim_{n
      \to \infty} p(\ot,\ot_n)\end{array}\right.\]
A standard compactness argument reveals that limits of order types
abound.  Indeed, for each element~$\ot_n$ in a sequence of order types,
the map~$\ot \in \sot \mapsto p(\ot,\ot_n)$ can be seen as a point in~${[0,1]}^\N$,
which is compact by Tychonoff's theorem. Any sequence of order types with sizes going
to infinity therefore contains a convergent subsequence.

\subsection{Problems and results}

We explore the application of the theory of limits of dense graphs to order
types in two directions. On one hand, the algebraic description of limits as
positive homomorphisms of flag algebras makes these limits amenable to
semi-definite programming methods. We implemented this approach for order types
and obtained numerical results.  On the other hand, the fact that measures
generally define limits of order types (\cref{l:lmu}) unveils stimulating
problems and interesting questions, of a more structural nature, on the
relation between measures and limits of order types.

\paragraph{Flag algebras of order types.}
The starting motivation for our work is a question raised by
Erd\H{o}s and Guy~\cite{ErGu73} in~1973 (see also~\cite{Erd84}):
``what is the minimum number~$\conv_k(n)$ of convex
$k$-gons in a set of~$n$ points in the plane?''. 
This falls within the scope of a general (and more conceptual than precise)
question of Sylvester:
``what is the probability that four points at random are in convex position?''.
Making sense of Sylvester's question implies defining a distribution on $4$-tuples of points,
and from the beginning of the~20th century several variants using
distributions coming from the theory of convex sets were investigated (\emph{e.g.}~uniform
or gaussian distributions on compact convex sets). 
For more background on this, the reader is referred to the survey by Ábrego, Fernández-Merchant
\&~Salazar~\cite{Xing} and to the book by Brass, Moser \&~Pach~\cite[Section~8]{problembook}.
We also point out that Sylvester's question is
actually related to several important conjectures in convex geometry~\cite[Chapter~3]{BGVV14}.
By a standard double-counting argument, one sees that~$\conv_k(n+1)\ge\frac{n+1}{n+1-k}\conv_k(n)$, so the
limiting density
\[c_k= \lim_{n \to \infty} \conv_k(n)/\binom{n}{k}\]
is well defined and equal to the supremum of this ratio for~$n\in\mathbf{N}$.
We apply the framework of flag algebras to order types and use the semi-definite method
to obtain lower bounds on~$c_k$ for~$k\in\{4,5,6\}$.
As it turns out, the literature around the computation of~$c_4$ is vast: not only does~$c_4$
correspond to the last open case of a relaxation of Sylvester's conjecture to all open
sets of finite area, but as discovered by Scheinerman and Wilf~\cite{ScWi94},
its value is determined by the asymptotic behaviour of the rectilinear crossing number of the complete graph,
which has been extensively investigated. For this particular case,
the best lower bound we could obtain
is~$c_4\ge0.37843917$, which falls short of the currently best known lower bound,
namely~$277/729>0.3799$. This better lower bound is obtained by plugging results of
Aichholzer \emph{et al.}~\cite{AGOR07} and~Ábrego \emph{et al.}~\cite{ACF+12,AFLS08} into
an expression of the rectilinear crossing number found independently by
Lovász \emph{et al.}~\cite{LVWW04} and by~Ábrego and Fernández-Merchant~\cite{AbFe05}.
The best currently known upper bound on~$c_4$ is~$c_4<0.380473$,
and is due to Fabila-Monroy and López~\cite{FaLo14}.
Nonetheless, our method allows us to strongly improve the known lower bounds
on~$c_5$ and on~$c_6$. 

\begin{proposition}\label{p:c5c6}
      We have $c_5 \ge 0.0608516$ and~$c_6 \ge 0.0018311$.
\end{proposition}

\noindent
To the best of our knowledge, prior to this work the
only known lower bounds on any constant~$c_k$ with~$k\ge5$ followed from a
general and important result of Erd\H{o}s and Szekeres~\cite{ErSz35}, \emph{via} a simple
double counting argument. One indeed sees that~$c_5\ge\frac{5!4!}{9!}>0.00793$
using that nine points in the plane must contain a convex pentagon
(a result attributed to Makai by Erdős and Szekeres~\cite{ErSz35}, the first published proof
being by Kalbfleisch, Kalbfleisch \& Stanton~\cite{KKS70}).
Similarly, as Szekeres and Peters~\cite{SzPe06} proved (using a computer-search)
that the Erd\H{o}s-Szekeres conjecture~\cite{ErSz35} is true for convex hexagons,
one can use that seventeen points in the plane must contain
a convex hexagon to infer that~$c_6\ge\frac{11!6!}{17!}>0.0000808$.
The best upper bounds that we are aware of on these numbers are~$c_5 \le 0.0625$ and~$c_6 \le 0.005822$.
We point out that Ábrego (personal communication) conjectured that~$c_5
=0.0625$.  We again refer the interested reader to the survey
and the book cited above~\cite{Xing,problembook}.

We prove~\cref{p:c5c6} by a reformulation of
limits of order types as positive homomorphisms from a so-called flag
algebra of order types into~$\R$ (see \cref{p:hom}); this
point of view allows a semidefinite programming formulation of the
search for inequalities satisfied by limits of order types.
Specifically, we argue that for any limit of order types~$\ell$
\[ \ell(\diamond_5) \ge 0.0608516 \qquad \text{and} \qquad
\ell(\diamond_6) \ge 0.0018311,\]
where~$\diamond_k$ is the order type of~$k$ points in convex position
for any positive integer~$k$.

On a related topic, we can mention a recent application of flag algebras
by Balogh, Lidick{\'y}, and {Salazar}~\cite{hill} to the (non rectilinear) crossing number
of the complete graph.
Their techniques differ from ours in that they use rotational systems instead of order types,
and they use results about the crossing number of the complete graph for small
numbers.

\medskip
We now turn to another aspect of our work using flag algebras of order types.
Probabilistic constructions often present
extremal combinatorial properties that are beyond our imagination, a textbook example being the lower
bound on Ramsey numbers for graphs devised by Erd\H{o}s~\cite{Erd47} in~1947.
Sampling order types of a given size \emph{uniformly} is of much interest to
test conjectures and search for extremal examples (see
\emph{e.g.}~\cite[p.~326]{problembook}). However the uniform distribution on
order types seems out of reach as suggested by the lack of closed formulas for
counting them, or heuristics to generate them, but we know of no formal
justification of the hardness of approximation of this distribution. As it
turns out, limits of order types can also be defined as families of probability
distributions on order types with certain internal consistencies (see
\cref{p:distrib}) and we exploit this interpretation to provide negative
results on certain sampling strategies.  Our second result obtained by the
semidefinite method of flag algebras indeed shows that a broad class of random
generation methods must exhibit some bias.

\begin{proposition}\label{p:imbalance}
For any limit of order types~$\ell$ there exist two order types~$\omega_1$
and~$\omega_2$ of size~$6$ such that
$\ell(\ot_1) > 1.8208 \ \ell(\omega_2)$. 
\end{proposition}

\noindent
This unavoidable bias holds, in particular, for the random generation
of order types by independent sampling of points from \emph{any}
finite Borel measure over~$\R^2$, see \cref{l:lmu}.

\paragraph{Representing limits by measures.}

Given a finite Borel measure~$\mu$ over~$\R^2$ and an order type~$\ot \in
\sot$, let~$\ell_\mu(\ot)$ be the probability that~$|\ot|$ random
points sampled independently from~$\mu$ have order type~$\omega$. It
turns out (\cref{l:lmu}) that $\ell_\mu$ is a limit of order
types if and only if every line is negligible for~$\mu$.  Going in the
other direction, it is natural to wonder if geometric measures could
serve as analytic representations of limits of order types, like
graphons for limits of dense graphs. This raises several questions, here are some that we investigate: does
every limit of order types enjoy such a realization? For those which do,
what does the set of measures realizing them look like?
In analogy with the study of realization spaces of order types~\cite{ARW17,DHH17,KaPa11,Stretchability}
we investigate how the weak topology on probability measures relates to the topology on limits. 

We provide an explicit construction of a limit of order types that cannot be
represented by any measure in the plane. For~$t \in (0,1)$ let~$\odot_t$ be the
probability distribution over~$\R^2$ supported on two concentric circles with
radii~$1$ and~$t$, respectively, where each of the two circles has
$\odot_t$-measure~$1/2$, distributed proportionally to the length on this
circle. Because every line is negligible for~$\odot_t$, \cref{l:lmu} ensures
that~$\ell_{\odot_t}$ is a limit of order types. We define~$\ell_\odot$ to be
the limit of an arbitrary convergent sub-sequence of~${(\ell_{\odot_{1/n}})}_{n
\in \N}$.

\begin{proposition}\label{t:norep}
  There exists no probability measure~$\mu$ such that $\ell_\mu =
  \ell_\odot$.
\end{proposition}

\noindent
The proof of \cref{t:norep} reveals that the
sequence~${(\ell_{\odot_{1/n}})}_{n \in \N}$ is in fact convergent,
so~$\ell_\odot$ is indeed an explicit example. Actually, an explicit
description (for instance based on the reformulation of~$\ell_{\odot}$ used in
the proof of \cref{l:manytriangles}) is possible, if tedious.

\medskip

Let us now turn our attention to a family of limits that has a particularly
nice realization space. Firstly, observe that nonsingular affine transforms
preserve order types. Thus, if~$\mu$ is a measure that does not charge any
line,~$f$ is a nonsingular affine transform, and~$\mu' = \mu \circ f^{-1}$ is
the push-forward of~$\mu$ by~$f$, then $\ell_{\mu'} = \ell_\mu$. More
generally, one can take~$f$ to be the restriction to~$\R^2$ of an ``adequate''
projective map (we spell out the meaning of ``adequate'' in
\cref{sec:spherical}). Our main contribution in this direction is a projective
rigidity result: we show that under natural measure theoretic conditions
on~$\mu$, the realization space of~$\mu$ is essentially a point.

\begin{theorem}\label{thm:proj}
  {Let~$\mu_1$ and~$\mu_2$ be two compactly supported measures
    of~$\mathbf{R}^2$ that charge no line and whose supports have
    non-empty interiors. If $\ell_{\mu_1}=\ell_{\mu_2}$, then there
    exists a projective transformation~$f$ such that $\mu_2\circ
    f=\mu_1$.}
\end{theorem}

\noindent

The converse of the statement of
Theorem~\ref{thm:proj} is not true: it can fail for instance if the line mapped to infinity
intersects the interior of~$\supp\mu$.
The approach we use to establish it actually
provides a necessary and sufficient condition, expressed in terms of
``spherical transforms'' as defined in Section~\ref{sec:spherical}.

\noindent
The third author conjectured~\cite{dJo17} that the condition on~$\mu$ can
be weakened.

\begin{conjecture}\label{c-remi}
  If~$\mu$ is a measure that charges no line and has support
 of Hausdorff dimension strictly greater than~$1$, then every
  measure that realizes~$\ell_{\mu}$ is projectively equivalent to~$\mu$.
\end{conjecture}

\noindent
It is easy to see that the conditions in \cref{c-remi} are necessary. 
Recalling that~$\diamond_k$ is the order type of~$k$ points in convex position,
the limit of order types that for every~$k$ gives probability~$1$ to~$\diamond_k$
can be realized by measures of Hausdorff dimension in~$(0,1)$. In particular,
these are pairwise projectively inequivalent. More interestingly, we build a
limit of order types~$\ell_E$ presenting projective flexibility, and which is
interesting from the combinatorial geometry perspective: indeed,~$\ell_E$
is a nearly optimal lower bound for the aforementioned
problem of Erd\H{o}s, witnessing that $\log c_k=\Theta(-k^2)$.

\begin{theorem}\label{l:decay-LE}
There exists a limit of order types~$\ell_E$ that can be realized by a measure
of Hausdorff dimension~$s$ for any~$s \in (0,1)$ for which~$\ell_E(\diamond_k)
    = 2^{-\frac{k^2}8+\bigo{k \log k}}$. Furthermore,~$\ell_E$ cannot be realized by a
measure that is absolutely continuous with respect to the Lebesgue measure.
\end{theorem}

As in the rigidity result, the relation to the Hausdorff dimension seems
fundamental: we observe that regular measures present a very different
asymptotic behavior with respect to the density of~$\diamond_k$.

\section{Flag algebras}

\subsection{Reformulation of limits}

Order types can be understood as equivalence classes of chirotopes
under the action of permutations (see below). As such, they form an
example of a \emph{model} in the language of
Razborov~\cite{Razborov:2007}, and the theory of limits of order types
is a special case of Razborov's work. This section provides a
self-contained presentation of the probabilistic and algebraic
reformulations of limits of order types.

\subsubsection{Probabilistic characterization}\label{sub-p}

Let~${(\ot_n)}_{n \in \N}$ be a sequence of order types converging to
a limit~$\ell$.  Let~$\ot \in \sot$, let~$k \ge |\ot|$ and let~$n_0$
be large enough so that $|\ot_n| \ge k$ for every~$n\ge n_0$. A simple
conditioning argument yields that for any~$n \ge n_0$,
\begin{equation}\label{eq:new}
      p(\ot,\ot_n) = \sum_{\ot' \in \sot_k}
p(\ot,\ot')p(\ot',\ot_n).
\end{equation}
Indeed, the probability that a random sample realizes~$\ot$ is the
same if we sample uniformly~$|\ot|$ points from a realization
of~$\ot_n$, or if we first sample~$k$ points uniformly from that
realization and next uniformly select a subset of~$|\ot|$ of these~$k$
points. It follows that any limit~$\ell$ of order types satisfies the
following \emph{conditioning identities}:
\begin{equation}\label{e:cond}
  \forall \ot \in \sot, \forall k \ge |\ot|, \quad \ell(\ot) = \sum_{\ot' \in \sot_k} p(\ot,\ot')\ell(\ot').
\end{equation}
As one notices, \cref{eq:new} yields that the conditioning identities
are equivalent to the (seemingly weaker) condition
\begin{equation}\label{e:cond1}
      \forall\ot\in\sot,\quad \ell(\ot)=\sum_{\ot'\in\sot_{|\ot|+1}}p(\ot,\ot')\ell(\ot').
\end{equation}

However, a stronger condition than~\cref{e:cond} is needed to
characterize limits of order types, as we illustrate in
Example~\ref{ex-stronger}.  The \emph{split
  probability}~$p\pth{\ot',\ot'';\ot}$, where~$\ot'$, $\ot''$
and~$\ot$ are order types, is the probability that a random partition
of a point set realizing~$\ot$ into two classes of sizes~$|\ot'|$
and~$|\ot''|$, chosen uniformly among all such partitions, produces
two sets with respective order types~$\ot'$ and~$\ot''$. (In
particular $p\pth{\ot',\ot'';\ot}=0$ if $|\ot| \neq |\ot_1|+|\ot_2|$.)
We provide a detailed proof of the following proposition, but the
reader familiar with the topic already sees the corresponding result
for dense graph limits~\cite[Theorem~2.2]{LoSz06}.

\begin{proposition}\label{p:distrib}
   A function~$\ell \colon \sot \to \R$ is a limit of order types if and
   only if 
\begin{equation}\label{eq:spm}
  \forall \ot', \ot'' \in \sot, \qquad \ell(\ot')\ell(\ot'')=\sum_{\ot\in
    \sot_{|\ot'|+|\ot''|}}p(\ot',\ot'';\ot)\ell(\ot),
\end{equation}
      and for every~$n \in \N$,
      the restriction~$\ell_{|\sot_n}$ is a probability distribution on~$\sot_n$.
\end{proposition}

\noindent
Before establishing \cref{p:distrib}, let us first point out that the product
condition~\eqref{eq:spm} implies the conditioning condition~\eqref{e:cond},
as is seen by taking for~$\ot''$ the (unique) order type of size~$1$.

\begin{proof}[Proof of \cref{p:distrib}]
We start by establishing the direct implication.
The fact that $\ell_{|\sot_n}$ is a probability distribution on~$\sot_n$ follows
from the definition of a limit. As for~\cref{eq:spm}, fix two order types~$\ot'$ and~$\ot''$ in~$\sot$.
Let~$\ot_n$ be a random order type sampled from~$\ell_{|\sot_n}$ where~$n \ge |\ot'|+|\ot''|$. Let
\[ \alpha_n = p(\ot',\ot_n)p(\ot'',\ot_n) \qquad \text{and} \qquad
\beta_n = \sum_{\ot\in
  \sot_{|\ot'|+|\ot''|}}p(\ot',\ot'';\ot)p(\ot,\ot_n).\]
Now, fix some point set~$P$ with order type~$\ot_n$.
By the definition, the value~$p(\ot,\ot_n)$ is the
probability that~$|\ot|$ points sampled uniformly from~$P$ have order
type~$\ot$.
Now, on one hand,~$\alpha_n$ equals the probability
that two independent events both happens: (i) that a set~$P'$ of~$|\ot'|$ random points chosen uniformly from~$P$ has order type~$\ot'$,
and (ii) that another set~$P''$ of~$|\ot''|$ random points
chosen uniformly from~$P$ has order type~$\ot''$.  On the other hand,
observe that~$\beta_n$ equals the probability that~(i) and~(ii) happen
\emph{and} that~$P'$ and~$P''$ are disjoint. The
difference~$|\alpha_n-\beta_n|$ is therefore bounded from above by the
probability that~$P'$ and~$P''$ intersect.  Bounding from above the
probability that~$P'$ and~$P''$ have an intersection of one or more
elements by the expected size of~$P' \cap P''$ yields that
\begin{equation}\label{eq:ab}
  \Bigl|p(\ot',\ot_n)p(\ot'',\ot_n) \;- \negthickspace\negthickspace\sum_{\ot\in
    \sot_{|\ot'|+|\ot''|}}\negthickspace p(\ot',\ot'';\ot)p(\ot,\ot_n)\Bigr| \le \Ex\left(|P' \cap P''|\right) = \frac{|\ot'||\ot''|}{|\ot_n|}.
\end{equation}
Taking~$n \to \infty$ in~\eqref{eq:ab} we see that~$\ell$ satisfies Equation~\eqref{eq:spm}.

Conversely, suppose that~$\ell$ satisfies both conditions.
For every integer~$n$, we pick a random order type~$r_n$ of size~$n^2$ according to~$\ell_{|\sot_{n^2}}$. We assert that
      \[ \forall \ot \in \sot, \qquad \Prob\left(\lim_{n \to \infty} p(\ot,r_n)
      \ne \ell(\ot)\right) = 0.\]
Since the number of order types is countable, we conclude that
the random sequence~${(r_n)}_{n \in \N}$ is convergent and has limit~$\ell$
with probability~$1$.

It remains to prove the assertion. Fix some~$\ot \in \sot$ and assume that $n \ge |\ot|$. Then
  \[ \Ex[p(\ot, r_n)] = \sum_{\ot'\in\sot_{n^2}}
      \Prob(r_n = \ot')p(\ot,\ot') = \sum_{\ot'\in\sot_{n^2}}
  \ell(\ot')p(\ot,\ot')=\ell(\ot),\]
and
    \[ \Var(p(\ot,r_n))= \Ex[p{(\ot, r_n)}^2] - {\Ex[p(\ot, r_n)]}^2 =
    \sum_{\ot' \in \sot_{n^2}} {p(\ot,\ot')}^2\ell(\ot') - {\ell(\ot)}^2.\]
By~\cref{eq:spm},
    \[ \Var(p(\ot,r_n))= \sum_{\ot' \in \sot_{n^2}} {p(\ot,\ot')}^2\ell(\ot') - \sum_{\ot''\in\sot_{2|\ot|}} p(\ot,\ot;\ot'')\ell(\ot'').\]
Since $2|\ot| \le |\ot|^2 \le n^2 = |r_n|$, \cref{e:cond} yields that
$\ell(\ot'') = \sum_{\ot' \in \sot_{n^2}} p(\ot'',\ot')\ell(\ot')$.
So we deduce that
  \[ \Var(p(\ot,r_n))= \sum_{\ot' \in \sot_{n^2}}\pth{{p(\ot,\ot')}^2-
    \sum_{\ot''\in\sot_{2|\ot|}}
    p(\ot,\ot;\ot'')p(\ot'',\ot')}\ell(\ot').\]
We observe that Inequality~\eqref{eq:ab} obtained above is valid in
general, and therefore
  \[ \Var(p(\ot,r_n)) \leq \sum_{\ot' \in \sot_{n^2}}
  \frac{|\ot|^2}{n^2} \ell(\ot') =  \frac{|\ot|^2}{n^2}.\]
By Chebyshev's inequality it thus follows that for any positive~$\varepsilon$,
  \[\Prob(|p(\ot, r_n)-\ell(\ot)|>\varepsilon)\leq \frac{|\ot|^2}{\varepsilon^2 n^2}.\]
  Therefore, for any $\ot$ and~$\varepsilon>0$, the sum~$\sum_{n \ge 1} \Prob(|p(\ot,
    r_n)-\ell(\ot)|>\varepsilon)$ is finite and hence the Borel-Cantelli lemma
  implies that with probability~$1$, only finitely many of
  the events~${\{|p(\ot, r_n)-\ell(\ot)|>\varepsilon\}}_n$ happen.
  Consequently, it holds with probability~$1$ that $\lim_{n \to \infty} p(\ot, r_n) = \ell(\ot)$.
\end{proof}

\cref{p:distrib} provides many examples of limits of order types.
\begin{coro}\label{l:lmu}
  Let~$\mu$ be a finite Borel measure over~$\R^2$. The map~$\ell_\mu\colon \ot
  \in \sot \mapsto p(\ot,\mu)$ is a limit of order types if and only
  every line is negligible for~$\mu$.
\end{coro}
\begin{proof}
Assume that~$\ell_\mu$ is a limit of order types and let~${(\ot_n)}_{n \in \N}$
be a sequence converging to~$\mu$. Let~$\therefore$ be the order type of
size~$3$. We have
  \[ p(\therefore, \mu) = \ell_\mu(\therefore) = \lim_{n \to
    \infty}p(\therefore, \ot_n) = 1,\]
  so three random points chosen independently from~$\frac1{\mu(\R^2)}\mu$ are aligned with probability~$0$,
      and consequently every line is negligible for~$\mu$.
  
  Conversely, assume that~$\mu$ is a measure for which every line is
  negligible. For every integer~$n \ge 3$ the restriction of~$\ell_\mu$ to~$\sot_n$
      is a probability distribution. Moreover, for every
  order type~$\ot\in\sot$ and every integer~$m>|\ot|$, we have
\[ \ell_\mu(\ot)=\sum_{\ot'\in \sot_{m}}p(\ot,\ot')\ell_\mu(\ot')\]
      by sampling~$m$ points from~$\mu$ and sub-sampling~$|\ot|$
      points among them uniformly.
  \cref{p:distrib} then implies that~$\ell_\mu$ is a limit of
  order types.
\end{proof}

We conclude this section with a simple example of a function~$\ell$ that is
not a limit of order types, and yet satisfies the conditioning
identities and restricts on every~$\sot_n$ to a probability
distribution.

\begin{example}\label{ex-stronger}
Let~$\ell_\circ$ be the limit where convex order types have
probability~$1$. Let $\ell_\odot$ be the limit defined before
\cref{t:norep} on page~\pageref{t:norep}. Set~$\ell$ to be~$\frac12
\pth{\ell_\circ + \ell_\odot}$.

First, note that if~$\ell_1$ and~$\ell_2$ are two limits of order
types and~$\ell_3$ is any convex combination of~$\ell_1$ and~$\ell_2$,
then the conditioning identities for~$\ell_1$ and~$\ell_2$ ensure that:
\[ \forall \ot\in\sot, \forall k>|\ot|, \quad \ell_3(\ot) = \sum_{\ot' \in \sot_{k}} p(\ot,\ot')\ell_3(\ot').\]
In particular, our function $\ell$ satisfies the conditioning
identities. Similarly, it is immediate to check that for every $n$,
the restriction of $\ell$ to $\sot_n$ is a probability distribution on
$\sot_n$.

We show that~$\ell$ does not satisfy~\eqref{eq:spm} by proving that
\[ \ell\pth{\flagFourTwo} \ell\pth{\bullet} = \frac{1}{32} \qquad \hbox{and} \qquad  \sum_{|\ot| = 5} p\pth{\flagFourTwo,\bullet\ ; \ot} \ell\pth{\ot} = \frac3{64}.\]
To establish these equalities, we compute certain values
of~$\ell_\odot$, which essentially boils down to computing two
probabilities: (i) that the number of points on the outer circle
equals that of the convex hull of the order type considered, and~(ii)
that the convex hull of these ``outer points'' contains the centre of
the circle.

To compute~(ii) requires to compute the probability that $k$ points
chosen uniformly on a circle have the circle's center in their convex
hull. Let us condition on the $k$ lines formed by the points and the
center of the circle (the lines being almost surely pairwise
distinct). Each point is selected uniformly from the two possible,
antipodal, positions on the corresponding line, and the initial choice
of the lines does not matter. So the probability reformulates as:
given $k$ pairs of antipodal points on the circle, how many of the
$2^k$ $k$-gons formed by one point from each pair contain the center?
For $k=3$ the answer is~$2$ out of~$8$ configurations, while for $k=4$
the answer is~$8$ out of~$16$ configurations. We thus infer that
\[ \ell_\odot\pth{\flagFourTwo} = 4\frac1{2^4}\frac14 = \frac1{16},  \quad \quad \qquad \ell_\odot\pth{\flagFiveThree} = \binom52\frac1{2^5}\frac14 = \frac{5}{64},\]
and
\[ \ell_\odot\pth{\flagFiveTwo} = 5\frac1{2^5} \frac12 = \frac5{64}.
\]

\noindent
It follows, on one hand, that
     \[ \ell\pth{\flagFourTwo} \ell\pth{\bullet} = \ell\pth{\flagFourTwo} = \frac12 \pth{\ell_\circ\pth{\flagFourTwo} + \ell_\odot\pth{\flagFourTwo}} = \frac12 \pth{0+\frac1{16}} = \frac1{32}.\]
and on the other hand, that
\[\begin{aligned}
\sum_{|\ot| = 5} p\pth{\flagFourTwo,\bullet\ ; \ot} \ell\pth{\ot} &=  \frac45 \ell\pth{\flagFiveThree} + \frac25 \ell\pth{\flagFiveTwo} + 0 \ell\pth{\flagFiveOne} \\
& = \frac25 \ell_\odot\pth{\flagFiveThree} + \frac15 \ell_\odot\pth{\flagFiveTwo}\\
    & = \frac25 \times \frac5{64} + \frac15 \times \frac5{64}\\
& = \frac3{64} \neq \ell\pth{\flagFourTwo} \ell\pth{\bullet}.
\end{aligned}\]
\end{example}

\subsubsection{Algebraic characterization}\label{sub-a}

Recall that by~\eqref{e:cond} every limit~$\ell$ of order types satisfies
\[
  \forall \ot \in \sot, \forall k \ge |\ot|, \quad \ell(\ot) = \sum_{\ot' \in \sot_k} p(\ot,\ot')\ell(\ot').
\]
Now, let~$\R\sot$ be the set of all finite formal linear combinations
of elements of~$\sot$ with real coefficients and consider the quotient
vector space
\[ \A = \R\sot/\Kern \qquad \text{where} \qquad \Kern =
\vect\Bigl\{\ot-\sum_{\ot' \in \sot_{|\ot|+1}}p(\ot,\ot')\ot' : \ot
  \in \sot\Bigr\}.\]
We define a product on~$\sot$ as follows:
\begin{equation}
\label{eq:flagproduct}
\forall \ot_1,\ot_2 \in \sot, \qquad  \ot_1 \times \ot_2 = \sum_{\ot \in \sot_{|\ot_1|+|\ot_2|}}
p(\ot_1,\ot_2;\ot)\ot,
\end{equation}
and we extend it linearly to~$\R\sot$. This extension is compatible with
the quotient by~$\Kern$ and therefore
turns~$\A$ into an algebra~\cite[Lemma 2.4]{Razborov:2007}.

An algebra homomorphism from~$\A$ to~$\R$ is \emph{positive} if
it maps every element of~$\sot$ to a non-negative real, and we define~$\Hom^+(\A, \R)$
to be the set of positive algebra homomorphisms from~$\A$ to
$\R$. Observe that any algebra homomorphism sends~$\cdot$, the
order-type of size one, to the real~$1$ as~$\cdot$ is the neutral element
for the product of order types.

\begin{proposition}[{\cite[Theorem~3.3b]{Razborov:2007}}]\label{p:hom}
  A map~$f \colon \sot \to \R$ is a limit of order types if and only if its
  linear extension is compatible with the quotient by~$\Kern$ and
  defines a positive homomorphism from~$\A$ to~$\R$.
\end{proposition}

We equip~$\A$ with a partial order~$\ge$, and write that $a_1 \ge a_2$
with~$a_1, a_2 \in \A$ if the image of~$a_1 - a_2$ under every
positive homomorphism is non-negative. The algebra~$\A$ allows us to
compute effectively with density relations that hold for \emph{every}
limit~$\ell$.

\begin{example}\label{ex:unlab}
  Let~$\cdot$ be the order type on one point,
  $\flagFourOne[4mm]$ and~$\flagFourTwo[4mm]$ the two order types of
  size four and~$\flagFiveOne[4mm]$, $\flagFiveTwo[4mm]$ and~$\flagFiveThree[4mm]$
      the three order types of size five, seen as elements of~$\A$. It follows
      from the definitions and from Equation~\eqref{e:cond} that
\begin{equation}\label{eq:flag1}
    \flagFourOne[4mm] = \flagFiveOne[4mm] + \frac35 \flagFiveTwo[4mm] + \frac15 \flagFiveThree[4mm] \quad \text{and} \quad \flagFiveOne[4mm] + \flagFiveTwo[4mm] + \flagFiveThree[4mm] = \cdot
\end{equation}
  Since for any limit of order types~$\ell$ we have~$\ell(\cdot)=1$,
  the above implies that
\[ \ell(\diamond_4) = \ell \pth{\frac15\flagFiveOne[4mm] + \frac15 \flagFiveTwo[4mm] + \frac15 \flagFiveThree[4mm]} +\ell \pth{\frac45\flagFiveOne[4mm] + \frac25 \flagFiveTwo[4mm]} \ge \frac15 \ell(\cdot) = \frac15 \]
  Using again Equation~\eqref{e:cond}, and the non-negativity of
  $\flagFiveThree[4mm]$ we obtain
  \[ \frac25 \flagFiveOne[4mm] \ge \flagFourOne[4mm] - \frac35 \ (\flagFiveOne[4mm] + \flagFiveTwo[4mm] + \flagFiveThree[4mm]) = \flagFourOne[4mm] - \frac35 \cdot \]
  and~$\ell(\diamond_5) \ge \frac52 \ell(\diamond_4)- \frac32$ for any
  limit of order types~$\ell$.
\end{example}

\subsubsection{Semi-definite method}

\cref{p:hom} allows to search for inequalities over~$\A$ by
semidefinite programming. Let us give an intuition of how this works
on an example. Here, we use the comprehensive list of all the order
types of size up to~$11$, which was made available by
Aichholzer~\footnote{\url{http://www.ist.tugraz.at/aichholzer/research/rp/triangulations/ordertypes/}}
based on his work with Aurenhammer and Krasser~\cite{aak-eotsp-01} on
the enumeration of order types. Throughout this paper, all non-trivial
facts we use without reference on order types of small size can be
traced back to that resource.

\begin{example}\label{ex:semidef}
  A simple (mechanical) examination of the~$6405$ order types of size~$8$ reveals that
      $p(\diamond_4,\ot) \ge 19/70$ for any~$\ot \in
  \sot_8$.  With Identity~\eqref{e:cond} this implies that
  $\flagFourOne[4mm] \ge 19/70 \ \cdot$ or equivalently that $c_4 \ge 19/70
  > 0.2714$.  Observe that for any~$C \in \A$ and any (linear
  extension of a) limit of order types~$\ell$ we have $\ell(C \times
    C)={\ell(C)}^2 \ge 0$ by~\cref{p:hom}.  We thus have at our
  command an infinite source of inequalities to consider to try and
  improve the above bounds. For instance, a tedious but elementary
  computation yields that
    \[ {\left( \frac6{25} \flagFourOne[4mm] -
    \frac{11}{125} \flagFourTwo[4mm] \right )}^2 +
  \frac{298819}{1093750}\sum_{\ot\in\sot_8} \ot = \sum_{\ot\in\sot_8}
  a_\ot \ot,\]
  where~$a_\ot \le p(\diamond_4,\ot)$ for every~$\ot \in \sot_8$. This
  implies that $\ell(\diamond_4) \ge 298819/1093750 > 0.2732$ for any
  limit of order types~$\ell$. The search for interesting combinations
  of such inequalities can be done by semidefinite programming.
\end{example}

\subsection{Improving the semidefinite method \emph{via} rooting and averaging}

The effectiveness of the semidefinite method for limits of graphs was
greatly enhanced by considering partially labelled graphs. We unfold
here a similar machinery, using some blend of order types and
chirotopes.
 
\paragraph{Partially labelled point sets, flags,  $\sigma$-flags and~$\A^\sigma$.}

A point set \emph{partially labelled} by a finite set~$\labelset$ (the
\emph{labels}) is a finite point set~$P$ together with some injective
map~$L\colon\labelset \to P$. It is written~$(P,\labelset,L)$ when we
need to make explicit the set of labels and the label map.
Two partially labelled point sets~$(P,\labelset,L)$ and~$(P',\labelset,L')$
\emph{have the same flag} if there exists a
bijection~$\phi\colon P \to P'$ that preserves both the orientation and the
labelling, the latter meaning that $\phi(L(i))=L'(i)$ for each~$i \in
\labelset$. The relation of having the same flag is an equivalence
relation, and \emph{a flag} is an equivalence class for this relation.
Again, we call any partially labelled point set a \emph{realization} of its
equivalence class, and the \emph{size}~$|\tau|$ of a flag $\tau$ is the
cardinality of any of its realizations.

A flag where all the points are labelled, \emph{i.e.} where
$|P|=|\labelset|$ in some realization $(P,\labelset,L)$, is a
\emph{$\labelset$-chirotope}. (When $\labelset = [k]=\{1,2,\dotsc,
k\}$ a $\labelset$-chirotope coincides with the classical notion of
chirotope.) Note that chirotopes correspond to types in the flag algebra terminology.
Discarding the non-labelled part of a flag~$\tau$ with
label set~$\labelset$ yields some $\labelset$-chirotope~$\sigma$
called the \emph{root} of~$\tau$. A flag with root~$\sigma$ is a
\emph{$\sigma$-flag} and by $\schi^\sigma$ we mean the set of~$\sigma$-flags.
The \emph{unlabelling} $\tau^\emptyset$ of a flag~$\tau$ with realization $(P,\labelset,L)$ is the order type of~$P$.

Let~$\labelset$ be a set of labels and~$\sigma$ a
$\labelset$-chirotope. We define densities and split probabilities for
$\sigma$-flags like for order types. Namely, let~$\tau$, $\tau'$ and~$\tau''$ be
$\sigma$-flags respectively realized by~$(P,\labelset,L)$, $(P',\labelset,L')$ and~$(P'',\labelset,L'')$.
The \emph{density} of~$\tau$ in~$\tau'$ is the probability that for a
random subset~$S$ of size~$|P|-|\labelset|$, chosen uniformly in~$P'
\setminus L'(\labelset)$, the partially labelled set~$(S \cup
L'(\labelset),\labelset,L')$ has flag~$\tau$. The \emph{split
  probability}~$p(\tau,\tau';\tau'')$ is the probability that for a
random subset~$S$ of size~$|P|-|\labelset|$, chosen uniformly in~$P''
\setminus L''(\labelset)$, the partially labelled sets~$(S \cup
L''(\labelset),\labelset,L'')$ and~$(P''\setminus S,\labelset,L'')$
have respective flags~$\tau$ and~$\tau'$.

We can finally define an algebra of~$\sigma$-flags as for order types.
We endow the quotient vector space
\[ \A^\sigma = \R\schi^\sigma/\Kern^\sigma \qquad \text{where} \qquad \Kern^\sigma =
\vect\Bigl\{\ot-\sum_{\ot' \in \schi^\sigma_{|\ot|+1}}p(\ot,\ot')\ot' : \ot
  \in \schi^\sigma\Bigr\}\]
with the linear extension of the product defined on~$\schi^\sigma$ by
$\tau \times \tau' = \sum_{\tau'' \in
  \schi^\sigma_{|\tau|+|\tau'|-|\sigma|}} p(\tau,\tau';\tau'')\tau''$.

\begin{example}\label{ex:lab}
    Here are a few examples to illustrate the notions we just introduced.
Letting~$\sigma$ be the unique $\labelset$-chirotope with~$|\labelset|=2$,
there are exactly seven $\sigma$-flags on four points, three with a convex hull of size~$4$ and four with
a convex hull of size~$3$. The densities of $\flagThreeLabTwo[6mm]$ into each of these seven $\sigma$-flags with~$4$ points indicated below.
\begin{align*}
    \label{eq:sigma1}
      p\left(\flagThreeLabTwo[7mm],\flagFourLabEight[7mm]\right) &= 1, & p\left(\flagThreeLabTwo[7mm],\flagFourLabThirteen[7mm]\right) &=0,\\
      p\left(\flagThreeLabTwo[7mm],\flagFourLabTen[7mm]\right) &= \frac12,&p\left(\flagThreeLabTwo[7mm],\flagFourLabNine[7mm]\right) &=\frac12,\\
      p\left(\flagThreeLabTwo[7mm],\flagFourLabFourteen[7mm]\right) &=1,&p\left(\flagThreeLabTwo[7mm],\flagFourLabFifteen[7mm]\right) &=0,\\
      p\left(\flagThreeLabTwo[7mm],\flagFourLabSixteen[8mm]\right) &=\frac12.
\end{align*}
Considering the quotient algebra, one sees that
\begin{align*}
    \flagThreeLabThree[8mm] &= \flagFourLabEleven[8mm] + \flagFourLabTwelveb[8mm]\\
\intertext{and}
    \flagFourLabTwelvea[8mm] &= \frac12\flagFiveLabOne[8mm] + \frac12\flagFiveLabTwo[8mm]
\end{align*}
\end{example}

\paragraph{Rooted homomorphisms and averaging.}

The interest of using the algebras~$\A^\sigma$ to study~$\A$ relies on three tools which
we now introduce. We first define an \emph{embedding} of a
$\labelset$-chirotope in an order type~$\ot$ to be a $\sigma$-flag with
root~$\sigma$ and unlabelling~$\ot$. We use \emph{random embeddings}
with the following distribution in mind: fix some point set realizing~$\ot$,
consider the set~$I$ of injections $f\colon\labelset \to P$ such
that~$(P,\labelset,f)$ is a $\sigma$-flag, assume that $I\neq\varnothing$,
choose some injection~$f_r$
from~$I$ uniformly at random, and consider the flag of~$(P,\labelset,f_r)$.
We call this the \emph{labelling} distribution on the
embeddings of~$\sigma$ in~$\ot$.

Next, we associate to any convergent sequence of order types~${(\ot_n)}_{n\in\N}$
with limit~$\ell$, and for every $\labelset$-chirotope~$\sigma$ such that $\ell(\sigma^\varnothing)>0$,
a probability distribution on~$\Hom^+(\A^\sigma,\R)$. For every~$n\in\N$,
the labelling distribution on embeddings of~$\sigma$ in~$\ot_n$ defines
a probability distribution~${\bf P}_{\bf n}^\sigma$ on mappings from~$\A^{\sigma}$
to~$\R$; specifically, for each embedding~$\theta_n$ of~$\sigma$ in~$\ot_n$ we
consider the map
\[ f_{\theta_n}: \left\{\begin{array}{rcl}\A^\sigma & \to & [0,1]\\ \tau & \mapsto
& p(\tau,\theta_n)\end{array}\right.\]
and assign to it the same probability, under~${\mathbf{P}}_{\bf n}^\sigma$, as the
probability of~$\theta_n$ under the labelling distribution. As~$\ell(\sigma^\varnothing)$
is positive, the fact that $p(\ot,\omega_n)$ converges as~$n \to \infty$ for every~$\ot\in\sot$
implies the weak convergence of the sequence~${({\mathbf{P}}_{\bf n}^\sigma)}_{n \in \N}$ to a Borel
probability measure~${\mathbf{P}}_{\ell}^\sigma$
on~$\Hom^+(\A^\sigma,\R)$~\cite[Theorems 3.12 and 3.13]{Razborov:2007}.  Moreover,
as $\ell(\sigma^\emptyset)$ is positive, the homomorphism induced by~$\ell$
determines the probability distribution~${\mathbf{P}}_{\ell}^\sigma$~\cite[Theorem
3.5]{Razborov:2007}.

We finally define, for every $\labelset$-chirotope~$\sigma$, an \emph{averaging}
(or downward) operator~$\llbracket\cdot\rrbracket_{\sigma}\colon \A^{\sigma} \to
\A$ as the linear operator defined on the elements of~$\tau\in \schi^\sigma$ by
$\llbracket{\tau}\rrbracket_{\sigma} = p_\tau^\sigma \cdot \tau^\emptyset$,
where~$p_\tau^\sigma$ is the probability that a random embedding of~$\sigma$
to~$\tau^\emptyset$ (for the labelling distribution) equals~$\tau$.

\begin{example}
  Here are a few examples of~$\sigma$-flags, where~$\sigma = 123$ is
  the CCW chirotope of size~$3$:
\begin{center} $\left\llbracket \flagThreeLab[8mm] \right\rrbracket_{123}
      = \frac12 \flagThree[6mm]$ \hfill $\left\llbracket \flagFourLabOne[8mm]
      \right\rrbracket_{123} = \frac16  \flagFourOne[6mm]$ \hfill $\left\llbracket
      \flagFourLabFive[8mm] \right\rrbracket_{123} = \frac18 \flagFourTwo[6mm]$
\end{center}
\end{example}

For every $\labelset$-chirotope~$\sigma$ and every limit of order types~$\ell$, we
have the following important identity~\cite[Lemma 3.11]{Razborov:2007}:
\begin{equation} \label{eq:flag:averaging} \forall \tau\in\A^\sigma,\quad
\ell\pth{\unlab{\tau}{\sigma}} = \ell\pth{\unlab\sigma\sigma} \int_{\phi^\sigma
    \in \Hom^+(\A^\sigma,\R)} \phi^\sigma(\tau) \diff{\mathbf{P}}_\ell^\sigma,
\end{equation}
which represents the fact that one can sample~$\unlab\tau\sigma$ by first picking a
copy of~$\sigma$ at random, and then, conditioning
on the choice of~$\sigma$, extend it to a copy of~$\tau$.
Equation~\eqref{eq:flag:averaging} in particular implies
that $\ell\pth{\left\llbracket{C^\sigma}\right\rrbracket_{\sigma}}\ge 0$
for any~$C^\sigma \in \A^\sigma$ such that $\phi^\sigma(C^\sigma)\ge 0$ almost
surely for $\phi^\sigma \in \Hom^+(\A^\sigma,\R)$, relatively to~${\bf
P}_\ell^\sigma$. It follows that for every limit of order types~$\ell$ and
every $\labelset$-chirotope~$\sigma$,
\begin{equation} \label{eq:flag:cauchyschwarz} \forall C^\sigma \in \A^\sigma,
\quad \ell\pth{\left\llbracket{\pth{C^\sigma}^2}\right\rrbracket_{\sigma}} \ge 0.
\end{equation}

\subsection{The semidefinite method for order types}

The operator $\llbracket \cdot \rrbracket_\sigma$ is linear, so for
every~$\phi \in \Hom^+(\A,\R)$, every~$A^\sigma_1,A^\sigma_2,\dotsc, 
A^\sigma_I \in \A^\sigma$, and every non-negative reals~$z_1,z_2,\dotsc,z_I$, we have
\[ \phi\left(\left\llbracket \sum_{i=1}^I {z_i
      \cdot{\left(A_i^\sigma\right)}^2}
  \right\rrbracket_{\sigma}\right)\ge 0.\]
Every real (symmetric) positive semidefinite matrix~$M$ of size~$n\times n$
can be written as $\sum_{i=1}^{n}\lambda_iu_i^Tu_i$ where~$\lambda_1,\dotsc,\lambda_n$
are non-negative real numbers and~$u_1,\dotsc,u_n$ orthonormal vectors of~$\mathbf{R}^n$.
It follows that for every finite set of flags~$S \subseteq \sot^\sigma$ and for every real
(symmetric) positive semidefinite matrix~$M$ of size~$|S| \times |S|$,
we have $\phi\left(\left\llbracket v_S^T M v_S
  \right\rrbracket_{\sigma}\right)\ge 0$, where~$v_S$ is the vector
  in~$\pth{\A^\sigma}^{|S|}$ whose $i$th coordinate equals the $i$th element
  of~$S$ (for some given order). This recasts the search for a
good ``positive'' quadratic combination as a semidefinite programming
problem.

Let~$N$ be an integer, $f = \sum_{\ot \in \sot_N} f_\ot \ot$ some
target function, and~$\sigma_1,\dotsc,\sigma_k$ a finite list of
chirotopes so that $|\sigma_i| \equiv N \mod 2$. For each $i
\in [k]$, let $v_i$ be the
$|\schi^{\sigma_i}_{(N+|\sigma_i|)/2}|$-dimensional vector with $i$th
coordinate equal to the $i$th element of~$\schi^{\sigma_i}_{(N+|\sigma_i|)/2}$. We look for a real~$b$ as large
as possible subject to the constraint that there are~$k$ real
(symmetric) positive semidefinite matrices~$M_1, M_2, \dotsc, M_k$,
where~$M_i$ has size~$|v_i| \times |v_i|$, so that
\begin{equation}
\label{eq:sdp}
 \forall \ot \in \sot_N, \quad f_\ot \ge a_\ot\qquad \text{where}
\qquad \sum_{\ot \in \sot_N} a_\ot \ot = \sum_{i\in[k]}
\left\llbracket v_i^T M_i v_i \right\rrbracket_{\sigma_i} + b
\sum_{\ot \in \sot_N} \ot.
\end{equation}

The values of the real numbers~$a_\ot$ are determined by~$b$, the entries of the
matrices~$M_1, M_2, \dotsc, M_k$, the splitting probabilities~$p(\tau',\tau'';\tau)$, where~$\tau', \tau'' \in
\schi^{\sigma_i}_{(N+|\sigma_i|)/2}$ and~$\tau \in
\schi^{\sigma_i}_N$, and the probabilities~$p_\tau^{\sigma_i}$, where~$\tau \in
\sot^{\sigma_i}_N$. Moreover, finding the maximum value of~$b$ and the entries of
the matrices~$M_i$ can be formulated as a semidefinite program.

\paragraph{Effective semidefinite programming for flags of order types.}

In order to use a semidefinite programming software for finding a
solution of programs in the form of~\eqref{eq:sdp}, it is enough to
generate the sets~$\sot_N$ and~$\schi^{\sigma_i}_N$, the split
probabilities~$p(\tau',\tau'';\tau)$, where~$\tau', \tau'' \in
\schi^{\sigma_i}_{(N+|\sigma_i|)/2}$ and~$\tau \in
\schi^{\sigma_i}_N$, and the probabilities~$p_\tau^{\sigma_i}$, where~$\tau \in
\sot^{\sigma_i}_N$.

We generated the sets and the values by brute force up to~$N=8$.  The
only non-trivial algorithmic step is deciding whether two order types,
represented by point sets, are equal. This can be done by computing
some canonical ordering of the points that turn two point sets with
the same order type into point sequences with the same chirotope. A
solution taking $\bigo{n^2}$ time was proposed by Aloupis et
al.~\cite{Oti}; the method that we implemented takes time~$\bigo{n^2\log n}$
and seems to be folklore (we learned it from Pocchiola and Pilaud).
For solving the semidefinite program itself, we used a library called
CSDP~\cite{Borchers:1999}.  The input data for CSDP was generated
using the mathematical software SAGE~\cite{sw:sage}.

\paragraph{Setting up the semidefinite programs.}
In the rest of this section we work with~$N=8$ and use chirotopes
labelled~$\sigma_1,\sigma_2, \dotsc,\sigma_{24}$ with~$\sigma_1$ being the
empty chirotope, $\sigma_2$ the only chirotope of size
\begin{wrapfigure}{l}{4cm}
\smallskip
\noindent
\begin{minipage}{3.5cm}
\typeFourOne[10mm]
\hspace{0.4cm}
\typeFourTwo[10mm]
\end{minipage}
\end{wrapfigure}
two, $\sigma_3$
and~$\sigma_4$ the two chirotopes of size~$4$ depicted on the left,
and~$\sigma_5,\dotsc,\sigma_{24}$ a fixed set of~$20$ chirotopes of size~$6$
so that $\sot_6 = \left\{\sigma_5^\emptyset,\dotsc,\sigma_{24}^\emptyset\right\}$.
Note that since $|\sot_6|=20$, what follows does not depend on the
choices made in labelling~$\sigma_5,\dotsc,\sigma_{24}$. The
vectors~$v_1,\dotsc,v_{24}$ described in the previous
paragraph for this choice of~$N$ and~$\sigma_i$'s have respective lengths~$2$,
$44$, $468$, $393$, $122$, $112$, $114$, $101$, $101$, $103$, $106$,
$103$, $103$, $120$, $102$, $108$, $94$, $90$, $91$, $91$, $95$, $95$,
$92$ and~$104$.

\paragraph{Computations proving \cref{p:c5c6} and~\cref{p:imbalance}.} 

We solved two semidefinite programs with the above choice of parameters
for~$f= \sum_{\ot \in \sot_8} p(\diamond_5,\ot)$ and~$f= \sum_{\ot \in
  \sot_8} p(\diamond_6,\ot)$ and obtained real symmetric positive
semidefinite matrices~$M_1,\dotsc,M_{24}$ and~$M_1',\dotsc,M_{24}'$ with
rational entries so that
\[ \sum_{\ot \in \sot_8} p(\diamond_5,\ot) \ot \ge \sum_{i\in[24]}
\llbracket v_i^T M_i v_i\rrbracket_{\sigma_i} + \frac{15715211616602583691}{258254417031933722624}
\sum_{\ot \in \sot_8} \ot,\]
and
\[\sum_{\ot \in \sot_8} p(\diamond_6,\ot) \ot \ge \sum_{i\in[24]}
\llbracket v_i^T M_i' v_i\rrbracket_{\sigma_i} + \frac{67557324685725989}{36893488147419103232}
\sum_{\ot \in \sot_8} \ot.\]
The lower bounds on~$c_5$ and~$c_6$ then follow from \cref{e:cond}.

Assume (without loss of generality) that $\sot_6 =
\{\ot_{6,1},\ot_{6,2},\dotsc,\ot_{6,20}\}$. Solving two semidefinite
programs, we obtained real symmetric positive semidefinite
matrices~$M_1,\dotsc,M_{24}$ and~$M_1',\dotsc,M_{24}'$ as well as non-negative
rational values~$d_1,\dotsc,d_{20}$ and~$d_1',\dotsc,d_{20}'$ so that
\begin{align*}
\sum_{j \in [20]} d_j \left(\ot_{6,j} - \frac1{32} \sum_{\ot \in \sot_8} \ot\right) + \sum_{i\in[24]} \llbracket v_i^T
    M_i v_i\rrbracket_{\sigma_i} &< 0\\
\shortintertext{and}
\sum_{j \in [20]} d_j' \left(- \ot_{6,j} + \frac1{18} \sum_{\ot \in \sot_8} \ot\right) + \sum_{i\in[24]} \llbracket
    v_i^T M_i' v_i\rrbracket_{\sigma_i} &< 0.
\end{align*}
  They respectively imply that there is no~$\ell \in \Hom^+(\A, \R)$ such that,
  $\ell(\ot) \ge 1/32$ for every~$\ot \in \sot_6$, or
  such that $\ell(\ot) \le 1/18$ for every~$\ot \in \sot_6$. Altogether
  this proves \cref{p:imbalance} with an imbalance bound of~$32/18 > 1.77$.
  The better bound of \cref{p:imbalance} is obtained by a refinement of this
  approach where the order types with minimum and maximum probability are
  prescribed; this requires solving over~$700$ semidefinite programs.

  The numerical values of the entries of all the matrices~$M_1,\dotsc,M_{24}$
  and coefficients~$d_1,\dotsc,d_{20}$ mentioned above can be downloaded from
  the web page \url{http://honza.ucw.cz/proj/ordertypes/}. In fact, the
  matrices~$M_1, \dotsc, M_{24}$ are not stored directly, but as an appropriate
  non-negative sum of squares, which makes the verification of positive
  semidefiniteness trivial. To make an independent verification of our
  computations easier, we created sage scripts called
  \texttt{verify\_prop*.sage}, available from the same web page.

\section{Representation of limits by measures}\label{s:measure}

\Cref{l:lmu} asserts that every probability (or finite Borel
measure) over~$\R^2$ that charges no line defines a limit of order
types. Going in the other direction, we say that a measure~$\mu$
\emph{realizes} a limit of order types~$\ell$ if $\ell_\mu = \ell$. We
examine here two questions: does every limit of order types enjoy such
a realization and, for those that do, what does the set of measures
realizing them look like? We answer the first question negatively in
\cref{s:nonrep}. We then give partial answers to the second
question in \cref{s:rigid} and~\cref{s:nonrigid}. Every measure
of~$\R^2$ or on~$\s^2$ that we consider is defined on the Borel $\sigma$-algebra.

\subsection{Spherical geometry}\label{sec:spherical}

If~$\mu$ is a measure that charges no line and~$f\colon \supp \mu \to
\R^2$ is an injective map that preserves orientations, then $\mu' = \mu
\circ f^{-1}$ is another measure that realizes the same limit
as~$\mu$. A map that preserves orientations must preserve alignments,
and therefore coincides locally with a projective map. Note, however,
that if we fix two points~$p,q \in \R^2$ and take a third point~$r$
``to infinity'' in directions~$\pm\vec{u}$, the orientation of the
triple~$(p,q,r)$ is different for~$+\vec{u}$ and for~$-\vec{u}$. The
appropriate geometric setting is therefore not projective, but
spherical.

\paragraph{The spherical model.}

For convenience, we write~$\{z>0\}$ for~$\sst{(x,y,z)\in\mathbf{R}^3}{z>0}$ and,
similarly, $\{z<0\}$ and~$\{z=0\}$ stand for~$\sst{(x,y,z)\in\mathbf{R}^3}{z<0}$ and
for~$\sst{(x,y,z)\in\mathbf{R}^3}{z=0}$, respectively.
Consider the standard identification of~$\R^2$ with the open upper
half of the unit sphere in $\R^3$ given by the bijection:
\[ \iota\colon \left\{\begin{array}{rcl}
\R^2 & \to & \s^2 \cap \{z>0\} \\ \pth{\begin{matrix}x\\y\end{matrix}} & \mapsto & \frac{1}{\sqrt{x^2+y^2+1}} \pth{\begin{matrix}x\\y\\1\end{matrix}}. \end{array}\right.\]
The map~$\iota$ sends the lines of~$\R^2$ to the intersection of
the half-space~$z>0$ with the great circles of~$\s^2$, that is circles cut out by
planes through the center of that sphere, other than the circle contained
in the plane~$z=0$. We can define the orientation of a triple~$(p,q,r)$ of
points of~$\s^2$ as
the sign of~$\det(p,q,r)$, so that it coincides, through~$\iota$, with the
planar notion. Naturally,~$\iota$ transports planar measures that charge no
line into spherical measures, supported on~$\s^2 \cap \{z>0\}$, that charge
no great circle.

\paragraph{Spherical transforms.}

Let~$X,Y \subset \R^2$. A function~$f\colon X \to Y$ is a \emph{spherical
  transform} if there exist a direct affine transform~$g$ and a
rotation~$h \in SO(\R^3)$ such that $h(\iota(X)) \subseteq \{z>0\}$
    and~$f = {(g \circ \iota^{-1} \circ h \circ \iota)}_{|X}$. In words,
spherical transforms are compositions of direct affine
transforms with a lifting of~$\R^2$ to the upper half of~$\s^2$, a
rotation that keeps the image of the lifting within that upper
half-sphere, then a backprojection from that upper sphere to~$\R^2$. Notice that~$f$
is defined over all of~$\R^2$ if~$h$ preserves the great circle~$\s^2
\cap \{z=0\}$, and on the open halfplane~$\iota^{-1} \pth{\{z>0\} \cap
h^{-1}(\{z>0\})}$ otherwise.

\begin{lemma}\label{l:st1}
  If~$\mu$ is a measure in~$\R^2$ that charges no line and~$f$ is a
  spherical transform defined over~$\supp \mu$, then~$\ell_\mu =
  \ell_{\mu \circ f^{-1}}$.
\end{lemma}
\begin{proof}
 The statement follows from the observation that~$f$ preserves
 orientations over~$\supp \mu$. To see this, let~$p,q,r \in \supp \mu$
    and let us decompose $f = {(g \circ \iota^{-1} \circ h \circ
    \iota)}_{|\supp \mu}$ where~$g$ is a direct affine transform and~$h \in
    SO(\R^3)$. Translations preserve orientation. Writing orientations as
 determinants reveals that direct linear transforms also preserve
 orientations. Altogether, the direct affine transform~$g$ thus also
 preserves orientations. It follows that

  \begin{align*}
    [f(p),f(q),f(r)] & =  [g \circ \iota^{-1}
      \circ h \circ \iota(p), g \circ \iota^{-1}
      \circ h \circ \iota(q), g \circ \iota^{-1}
      \circ h \circ \iota(r)]\\
     & =   [\iota^{-1}
       \circ h \circ \iota(p), \iota^{-1}
       \circ h \circ \iota(q), \iota^{-1}
       \circ h \circ \iota(r)]\\
     & =  \sign \det \pth{\begin{matrix} h \circ \iota(p) & h \circ \iota(q) & h \circ \iota(r) \end{matrix}} \\
     & =  \sign \pth{\det (h) \det \pth{\begin{matrix} \iota(p) & \iota(q) & \iota(r) \end{matrix}}}\\
     & =  \sign \pth{\det (h)} [p,q,r],
   \end{align*}
and~$f$ preserves orientations because~$h$ is direct.
\end{proof}

\begin{lemma}\label{l:st2}
 Let~$C$ be a closed convex set and~$L$ a line in~$\R^2$. If
 $d(C,L)>0$ and~$C$ contains no ray parallel to~$L$, then there exists
 a spherical transform~$f$ defined over~$C$ such that $f(C)$ is
 compact.
\end{lemma}
\begin{proof}
  Let~$X$ be the open halfpane bounded by~$L$ that contains~$C$.
  The image~$\iota(L)$ coincides with the intersection of a great
  circle~$\Gamma$ with~$\{z>0\}$. Moreover, $\iota(X)$ is contained in
  one of the hemispheres bounded by~$\Gamma$. Let~$h \in SO(\R^3)$ be a
  rotation that maps~$\Gamma$ to~$\s^2 \cap \{z=0\}$ and maps the interior
  of~$\iota(X)$ to (a subset of) the region~$\{z > 0\}$. We set~$f =
    {(\iota^{-1} \circ h \circ \iota)}_{|X}$.

  \bigskip
  
  We argue that if~$f(C)$ is unbounded, then~$d(C,L)=0$ or~$C$
    contains a ray parallel to~$L$. Let~${(p_n)}_{n \in \N} \subset C$
    such that ${(f(p_n))}_{n \in \N}$ is unbounded. Since~$\s^2$ is
  compact, up to taking a subsequence we can assume that~$\{h
  \circ\iota(p_n)\}$ converges. Its limit~$u$ must belong to~$\{z=0\}$,
  as otherwise $\iota^{-1}(u) \in \R^2$, which would
  contradict the assumption that $\{f(p_n)\}$ is unbounded. This means
  that $(\iota(p_n))$ converges to a limit~$u' =h^{-1}(u)\in
    \Gamma$. If~$u' \in \iota(L)$, then~$d(C,L)=0$. If~$u'\notin\iota(L)$,
    then $u' \in\Gamma \cap \{z=0\}$ and we can write~$u' = (u_x:u_y:0)$ where $\tilde u' =
  (u_x,u_y)$ is a unit direction vector for~$L$ to which the vectors
    ${\left(\frac{p_1p_n}{\|p_1p_n\|}\right)}_{n \in \N}$ converge. For
  any positive real number~$t>0$ and for any large enough integer~$n$ we have~$\|p_1p_n\| >t$, so the
  segment~$p_1p_n$ contains a point~$q_n$ with $\|p_1q_n\| = t$. The
  sequence $\{q_n\}$ is contained in~$C$ by convexity, and converges
  to~$p_1 + t \tilde u$. Since~$C$ is closed, this limit is also in
  $C$. The ray~$p_1+\R^+\tilde u'$ is thus in~$C$, which contradicts
  the assumption that~$C$ contains no ray parallel to~$L$.
  We deduce that~$f(C)$ is bounded and thus, altogether, the
  assumptions on~$L$ and~$C$ guarantee that~$f(C)$ is compact.
\end{proof}

\subsection{A limit that is not representable}\label{s:nonrep}

For convenience, let us repeat the definition of~$\ell_\odot$ from the
introduction. For~$t \in (0,1)$ let~$\odot_t$ be the probability
distribution over~$\R^2$ supported on two concentric circles, with
respective radii~$1$ and~$t$ where each of the two circles has
$\odot_t$-measure~$1/2$, distributed proportionally to the length on
that circle. We define~$\ell_{\odot_t}$ to be the limit of order types
associated to~$\odot_t$. We define~$\ell_\odot$ to be the limit of an
arbitrary convergent sub-sequence of~${(\ell_{\odot_{1/n}})}_{n \in \N}$.

We use the following two facts about measures, limits of order
types, and sequences of point sets. Let~$\mu$ be a measure of the plane,
and~$P_n$ a set of~$n$ points of the plane sampled i.i.d.\ according to~$\mu$.
The \emph{empirical measure} associated to~$P_n$ is the measure~$\mu_{P_n}$ defined
on all $\mu$-measurable subsets~$A$ of~$\mathbf{R}^2$ by
\[
    \mu_{P_n}(A)=\frac1n|P_n\cap A|.
\]

\begin{lemma}\label{l:Pnmu}
  Let~$\mu$ be a measure that does not charge lines and for~$n \ge 1$,
  let~$P_n$ be a set of points sampled independently from~$\mu$,
    with $|P_n| \underset{n \to \infty}{\to} \infty$.
  The empirical measure of~$P_n$ almost-surely weakly converges to~$\mu$.
\end{lemma}

\noindent
The next lemma follows from ideas similar to those used in the proof of Proposition~\ref{p:distrib}.

\begin{lemma}\label{l:Pnlim}
    Let~$\mu$ be a measure that does not charge lines. If ${(P_n)}_{n\in\mathbf{N}}$ is a
  sequence of point sets with $|P_n| = n^2$ and whose empirical measures converge to~$\mu$ as~$n$ goes to infinity,
  then the order type of~$P_n$ converges to~$\ell_\mu$ as~$n$ goes to infinity.
\end{lemma}

\subsubsection{Reduction to compact support}

Define the \emph{peeling depth} of
an order type~$\omega$ as the largest integer~$k$ such that $P_k \neq
\emptyset$ in a sequence~${\{P_i\}}_{i\in\N}$ where~$P_1$ realizes~$\omega$ and
$P_{i+1} = P_i \setminus \conv(P_i)$ for each~$i \ge 1$. For
instance, $k$ points in convex position have peeling depth~$1$, and
any~$k$ points chosen from the support of~$\odot_t$ have peeling depth
at most~$2$.

\begin{lemma}\label{l:pdc}
  Fix~$k \in \N$ and let~$\ell$ be a limit of order types such that
  $\ell(\omega) = 0$ for any order type~$\omega$ of peeling depth greater
  than~$k$. If there exists a probability measure~$\mu$ realizing~$\ell$,
      then there exists one with compact support.
\end{lemma}
\begin{proof}
  Let~$\mu$ be a probability measure that realizes~$\ell$ and let~$C =
  \conv \supp \mu$. If an order type~$\omega$ can be realized by
  points in the support of~$\mu$, then $\ell_\mu(\omega) > 0$. Indeed,
  for a positive and sufficiently small~$\varepsilon$, any perturbation of the points
  of amplitude at most~$\varepsilon$ also realizes~$\omega$. It follows
  that the support of~$\mu$ cannot contain any finite subset with
  peeling depth~$k+1$.

We first argue that $C \neq \R^2$. We proceed by contradiction, so
suppose $C = \R^2$. A theorem of Steinitz~\cite{Helly-interior}
asserts that for any subset~$A \subseteq \R^2$, any point in the
interior of~$\conv A$ belongs to the interior of the convex hull of at
most four points of~$A$. For any point~$p \in \R^2$, there is thus a
quadruple~$Q(p)\subset \supp \mu$ such that $p$ belongs to the
interior of~$\conv Q(p)$. We now define~$Y_1 = \{p_1\}$ where~$p_1$ is
an arbitrary point of~$\supp \mu$, and~$Y_{i+1} = \cup_{p \in Y_i}
Q(p)$ for each~$i\ge 1$. The set $\cup_{i=1}^{k+1} Y_i$ is contained
in the support of~$\mu$ and has peeling depth at least~$k+1$, a
contradiction. Consequently, $C$ cannot be~$\R^2$.

We can also rule out the case where~$C$ is a halfplane. Suppose the
contrary and let~$L$ be the line bounding this halfplane. We can
follow the argument showing that $C \neq \R^2$, except if at some
stage, some point of~$Y_i$ belongs to~$L$. Assume that this is the case
for~$p \in Y_i$, and note that we can assume that $i \ge 2$. There
exists~$p' \in Y_{i-1}$ such that~$p \in Q(p')$; remark that the
peeling depth of~$\cup_{j=1}^i Y_j$ remains at least~$i$ if we
replace~$p$ in~$Y_i$ by any point~$q$ in~$\supp \mu$ such that $p'$
belongs to the interior of~$\conv Q(p') \setminus \{p\}\cup
\{q\}$. There exists~$\varepsilon_1>0$ such that for any~$z \in
B(p,\varepsilon_1)$, the point~$p'$ belongs to the interior of~$\conv
Q(p') \setminus \{p\}\cup \{z\}$. Set~$\tau =
\mu(B(p,\varepsilon_1))$. Since~$\mu(L)=0$, there exists~$\delta>0$ such
that the set of points within distance~$\delta$ of~$L$ has measure at
most~$\varepsilon/2$. The part of~$B(p,\varepsilon_1)$ at distance at
least~$\delta$ from~$L$ therefore has measure at least~$\varepsilon_1/2$.
A dichotomy argument produces a point~$q$ in that part that is in the
support of~$\mu$. We can replace~$p$ by~$q$ in~$Y_i$ and continue with
the construction.  It follows that $C$ cannot be a halfplane.

  For any point~$p \in \partial C$ there is a line supporting~$C$ and
  containing~$p$. Thus, since~$C$ is neither~$\R^2$ nor a halfplane,
  $C$ is contained in the intersection of two non-parallel, closed,
  halfplanes~$H_1$ and~$H_2$. Let~$L$ be a line disjoint from~$H_1 \cap
  H_2$ that is parallel to neither~$\partial H_1$ nor~$\partial
  H_2$. Note that \cref{l:st2} applies, so there is a spherical
  transform~$f$ defined over~$\supp \mu$. \Cref{l:st1} ensures
  that the pushback of~$\mu$ by~$f$ is a measure with compact support
  that defines the same limit of order types as~$\mu$.
\end{proof}

\begin{coro}\label{c:comp}
  If~$\ell_\odot$ is realizable by a measure, then it is realizable by one with compact support.
\end{coro}
\begin{proof}
  For any~$t \in (0,1)$, if an order type~$\omega$ has peeling depth at
  least~$3$, then $\ell_{\odot_t}(\omega)=0$. It follows that~$\ell_\odot$ vanishes
  on every order type of peeling depth three or more and therefore \cref{l:pdc} yields
    the conclusion.
\end{proof}

\subsubsection{Excluding realizations with compact support}

Assume by contradiction that there exists a measure~$\mu$ that
realizes~$\ell_\odot$; in particular, $\mu$ charges no line. Let~$P_n$
be a set of~$N=n^2$ random points sampled independently from~$\mu$. By
\cref{l:Pnmu}, the empirical measures of~${(P_n)}_n$ weakly converge to~$\mu$ and by \cref{l:Pnlim},
their order types converge to~$\ell$. We write~$C_n = P_n \cap \partial \conv P_n$ for the extreme
points of~$P_n$ and~$I_n = P_n \setminus C_n$ for the remaining points.

\begin{lemma}\label{l:manytriangles}
  With high probability, $0.49N \le |C_n| \le 0.51N$ and at least~$f(N)$
  edges of~$\conv P_n$ form, with another vertex of~$C_n$, a
  triangle that contains $I_n$, where $\lim_{N \to \infty} f(N) =
  +\infty$.
\end{lemma}

\noindent
Before we prove \cref{l:manytriangles}, let us see how to use it
by establishing \cref{t:norep}.

\begin{proof}[Proof of \cref{t:norep}]
  Suppose, contrary to the statement, that there exists a measure~$\mu$ that
  realizes~$\ell_\odot$. In particular,~$\mu$ charges no line. By
  \cref{c:comp}, we can assume~$\mu$ to have compact support.
  Let~$\alpha$ be the length of the boundary of~$\conv \supp
  \mu$. Note that~$\alpha<\infty$ since~$\mu$ has compact
  support. Since~$\conv P_n$ is contained in~$\conv \supp \mu$, the
  boundary of~$\conv P_n$ also has length at most~$\alpha$ (this
  follows for instance from the Cauchy-Crofton formula~\cite[Chapter~3,
    $\mathsection~2$]{santalo1953introduction}). \Cref{l:manytriangles} implies that
  there is an edge of~$\conv P_n$ of length at most~$\varepsilon =
  \bigo{\frac{\alpha}{f(N)}}$ that forms, together with another vertex
  of~$C_n$, a triangle containing~$I_n$. Let~$L_n$ be the line
  supporting one of the edges of this triangle that is not
  on~$\partial \conv P_n$. With high probability, at
  least~$0.49|P_n|$ points of~$P_n$ lie within distance at
  most~$\varepsilon$ of~$L_n$. The set of lines intersecting~$\conv \supp
    \mu$ is compact, so the sequence~${(L_n)}_n$ contains a subsequence~${(L_{i(n)})}_n$
  converging to a line~$L$.  Since the empirical measures
    of~${(P_{i(n)})}_n$ converge to~$\mu$, we deduce that $\mu(L) \ge 0.49$, a
  contradiction.
\end{proof}

It remains to prove \cref{l:manytriangles}.

\begin{proof}[Proof of \cref{l:manytriangles}]
  Let us give another model of random point sets whose order-type is
  distributed like~$P_n$. Let~$\Circ_t$ be the circle of radius~$t$
  in~$\R^2$ centered at the origin. We let~$P_n^{1/2}$ be a set
  of~$N=n^2$ random points sampled independently
  from~$\odot_{\frac12}$ and for~$t\in\left(0,\frac12\right)$ we
  let~$P_n^t$ be the point set obtained from~$P_n^{1/2}$ by scaling
  the points on~$\Circ_{1/2}$ from the origin by a
  factor~$2t$. Observe that~$P_n^t$ is distributed like~$N$ points
  sampled independently from~$\odot_t$. Moreover, almost surely, there
  exists (a random variable) $t^* >0$ such that the order types
  of~$P_n^t$ are identical for~$t\in\left(0,t^*\right)$. As~$t\to 0$,
    the sequence~${\left(\ell_{\odot_t}\right)}_t$ converges
  to~$\ell_{\odot} = \ell_\mu$, so the order type of~$P_n^{t^*}$ is
  distributed like $P_n$.

  \bigskip
  
  Let us define~$C_n^{t^*} = P_n^{t^*} \cap \partial \conv P_n^{t^*}$
  and~$I_n^{t^*} = P_n^{t^*} \setminus C_n^{t^*}$. We assert that with
  high probability
    \begin{enumerate}
        \item[(a)] $0.49N \le |P_n^{t^*} \cap \Circ_1| \le 0.51N$; and 
        \item[(b)] $C_n^{t^*} = P_n^{t^*} \cap \Circ_1$.
    \end{enumerate} 
    Since the order types of~$P_n$ and~$P_n^{t^*}$ are identically distributed,
    this claim implies in particular the first part of the statement.

  Assertion~(a) follows from Hoeffding's inequality because the
  repartition of the points of~$P_n^{t^*}$ between~$\Circ_1$ and~$\Circ_{t^*}$ follows a binomial law with parameter~$1/2$. To prove
  Assertion~(b), let us bound the probability that there exists a point
  from~$\Circ_{t^*}$ on~$\partial \conv P_n^{t^*}$. If such a point
  exists, then there is a line through that point that cuts an arc of~$\Circ_1$
      with no point of~$P_n^{t^*}$. Any line through a point of~$\Circ_{t^*}$ cuts
      an arc of~$\Circ_1$ with length at least a
  fraction~$\alpha(t^*)$ of the length of~$\Circ_1$.
    We can now cover~$\Circ_1$ by smaller intervals of
  length~$\alpha(t^*)/2$, and we observe that $P_n^{t^*}$ must miss at
  least one of these smaller intervals. For~$t^*$ small enough,~$8$
  smaller intervals suffice, and~$C_n^{t^*} = P_n^{t^*} \cap \Circ_1$
    with probability least~$1-8 \pth{\frac{15}{16}}^{N}$. This proves our
  assertion.

  \medskip

  Let us now turn our attention to the second part of the
  statement of the lemma. We write~$C_n^{t^*} = \{p_1,p_2, \dotso\}$ where the
  points are labelled in clockwise order along~$\partial \conv P_n^{t^*}$,
  starting in an arbitrary place. We label the points cyclically,
  setting~$p_{i+|C_n^{t^*}|} = p_i$. Almost surely, for every~$i$ the
  antipodal of~$p_i$ is not in~$C_n^{t^*}$ so there exists a unique
  edge~$p_jp_{j+1}$ such that~$p_ip_jp_{j+1}$ contains~$I_n^{t^*}$. We
  set~$j(i)=j$. Our task is to bound from below the number of distinct
  edges~$p_{j(i)}p_{j(i)+1}$.

  Let us associate to every point~$p_i$ the random variables~$s_i =
 |i-j(i)|$. We assert that with high probability
\begin{equation}\label{eq:conc} \frac{|C_n^{t^*}|}2 - N^{\frac23} \le \min_i s_i \le \max_i s_i  \le \frac{|C_n^{t^*}|}2 +
  N^{\frac23}.\end{equation}
    Observe that~\eqref{eq:conc} implies that~$j(i)$ is different for any two
    indices~$i \in \{0,2.1N^{2/3}, 4.1N^{2/3}, \dotsc, (k+0.1)N^{2/3}\}$ where~$k\in\mathbf{N}$
    such that $(k+0.1)N^{2/3}\le0.49N<(k+1.1)N^{2/3}$. It follows
  that with high probability~$j(i)$ takes at least~$\Omega(N^{1/3})$
  different values, proving the second statement.

  It remains to establish~\eqref{eq:conc}. The random
  variable~$s_i$ counts the number of points on one of the halfcircles
  of~$\Circ_1$ bounded by~$p_i$ and its antipodal. Consider a
  different labeling~$C_n^{t^*} = \{q_1,q_2, \dotso\}$ where~$q_1$ is
  an arbitrary point of~$C_n^{t^*}$ and where the lines through the
  origin and~$q_1, q_2, \dotso$ are in clockwise order. Each~$s_i$
  counts the number of~$q_j \neq p_i$ on the clockwise half-circle
  from~$p_i$ to its antipodal. Observe that the position of the lines
  spanned by the points~$q_i$ and their antipodals are irrelevant; in fact,
  fixing the lines arbitrarily and picking each~$q_i$ uniformly among
  the two candidate points on its assigned line leads to the same
  distribution of the variables $s_i$. A purely combinatorial,
  equivalent, description is thus the following. Let~$W = w_1w_2\dotso
  w_{2k}$ be a random circular word on~${\{0,1\}}^{2k}$, where~$k =
  |C_n^{t^*}|$, $w_{i+k} = 1-w_i$ and the letters~$w_1, w_2, \dotsc,
  w_k$ are chosen independently and uniformly in~$\{0,1\}$. The random
  variable $s_i$ has the same distribution as~$w_{i} + w_{i+1} +
  \dotsb + w_{i+k-1}$. Hoeffding's inequality and a union bound then
  yield that with high probability,~$\max_i s_i$ is at most~$\frac{k}2 +
  \littleo{k}$ (and, by symmetry, that~$\min_i s_i$ is at least~$\frac{k}2 -
  \littleo{k}$).
\end{proof}

\subsection{The rigidity theorem}\label{s:rigid}

In this section, we prove Theorem~\ref{thm:proj}, which states that probability
distributions giving rise to the same limit of order types are projectively
equivalent, provided that their support have non-empty interiors.  We proceed
by adapting the notion of kernels to three-variable functions.  This notion
generalizes the notion of geometric limit firstly by allowing probability
measures of any topological space~$J$ instead of probability measures on~$\R^2$
and secondly by using any real-valued function of~$J^3$ instead of the
chirotope of the plane. Our main tool\footnote{This result and its proof are
direct adaptations of similar results on two-dimensional kernels and graphons;
the presentation reaches the proof of Proposition~\ref{prop:big} as directly as
possible, and we refer to the monograph of Lovász~\cite{Lov12} for (a lot) more
details on two-dimensional kernels.} on kernels is \cref{cor:pure}, which
asserts that two kernels have same sampling distribution if and only if they
are isomorphic. We introduce the notions related to kernels and state
\cref{cor:pure} in Section~\ref{s:overviewrigid}, but defer all proofs to
Section~\ref{sec:proof-kernels}, after we have examined kernels defined from
geometric probabilities and order types and deduced~\cref{thm:proj}
(Section~\ref{s:dWgeometric}).

\subsubsection{Rigidity for kernels (overview)}\label{s:overviewrigid}

A \emph{kernel} is a triple~$(J,\mu,W)$ where~$(J,\mu)$ is a
probability space and~$W$ is a measurable map from~$J^3$
to~$\mathbf{R}$ such that $\sup_{J^3}|W| < \infty$. For the sake of
readability,~$W$ can stand for a whole kernel~$(J,\mu,W)$ when there
is no ambiguity about the set and the measure involved. We are, in
particular, interested in kernels based on chirotopes, that is, where
$J=\R^2$, $\mu$ is a measure that charges no line, and~$W$ is a map
to~$\{-1,0,1\}$ given by the orientation.

Our first ingredient is a rigidity theorem for kernels
(Theorem~\ref{cor:pure}), which asserts that if two kernels have the same
density functions (the analogue, for general kernels, of the limit of order
types~$\ell_\mu$ associated with a measure~$\mu$), then they are essentially
equal. We present here the main notions and results, and postpone the proofs to
Section~\ref{sec:proof-kernels}.

\paragraph{Density function.} 

A \emph{step function} over a probability space~$(J,\mu)$ is a
kernel~$(J,\mu,U)$ for which there is a partition~$(P_1,\dotsc,P_m)$
of~$J$ such that the function~$U$ is constant on~$P_i\times P_j\times
P_k$ for every~$(i,j,k)\in{[m]}^3$.  Since the set~$U^{-1}(\{x\})$ is
measurable for every~$x\in\R$, we may always assume that~$P_i$ is
measurable for~$i\in[m]$.

Let~$(J,\mu,W)$ be a kernel and~$n$ a positive integer. Any
$n$-tuple~$S=(x_1,\dotsc,x_n) \in J^n$ defines a function
\[ \mathbf{H}(W,S)\colon \left\{\begin{array}{rcl}{[n]}^3 & \to& \mathbf{R} \\ (i,j,k) & \mapsto & W(x_i,x_j,x_k).\end{array}\right.\]
    We observe that if~$\mu_n^u$ is the counting measure on~$[n]$,
    then~$([n],\mu_n^u,\mathbf{H}(W,S))$ is a step function for the partition
    of~$[n]$ into singletons.\label{page-step} 
We let~$\mathbf{H}(W,n)$ be the random step
function~$\mathbf{H}(W,S)$, where~$S$ is a $\mu^n$-random tuple. 
We define~$\mathcal{H}_n$ to be~$\{H\colon{[n]}^3\to\mathbf{R}\}$.
For every~$H\in\mathcal{H}_n$, the
\emph{density~$t(H,W)$} of~$H$ in~$W$ is the probability that
$\mathbf{H}(W,n) = H$.
We set~$\mathcal{H}=\bigcup_{n\ge1}\mathcal{H}_n$.

\paragraph{Norms and distance.}

A kernel~$(J, \mu, W)$ has associated~$L_1$- and~$L_\infty$-norms
respectively defined by
\[
\norm{W}_1 = \int_{J^3} \abs{W(x,y,z)}\diff\mu(x)\diff\mu(y)\diff\mu(z)
\qquad \text{and} \qquad
\norm{W}_\infty=\sup_{J^3}\abs{W}.
\]
The~$L_1$- and~$L_\infty$-norms define distances only between kernels
with the same underlying probability space. To define a distance
between any two kernels $(J_1,\mu_1,W_1)$ and~$(J_2,\mu_2,W_2)$, we
first map them both to a common probability space by
measure-preserving\footnote{If~$(J_1,\mu_1)$ and~$(J_2,\mu_2)$ are two
  probability spaces, a map~$\rho\colon J_1\to J_2$ is \emph{measure
    preserving} if for every $\mu_2$-measurable set~$A'\subseteq J_2$,
  the set $A=\rho^{-1}(A')$ is $\mu_1$-measurable and
  $\mu_1(A)=\mu_2(A')$.}  maps (like the Gromov-Hausdorff distance
between metric spaces). Specifically, we define
\[
      \dist(W_1,W_2)= \inf \norm{W_1^{\phi_1} - W_2^{\phi_2}}_1
\]
where the infimum is taken over all choices of a probability
space~$(J,\mu)$ and a pair of measure-preserving maps~$\phi_i \colon J
\to J_i$ for~$i\in\{1,2\}$, and where~$W^{\phi}$ is defined by
$W^{\phi}(x,y,z)=W(\phi(x),\phi(y),\phi(z))$.
This infimum does not change if we assume that~$\mu$ is a coupling
measure on~$J=J_1 \times J_2$ and that~$\phi_i$ is the natural
projection of~$J$ on~$J_i$ for~$i\in\{1,2\}$.  Further, a theorem by
Janson~\cite[Theorem~2]{Jan09} ensures that, up to a
measure-preserving map, we can also assume that~$J_i$ is~$[0,1]$
and~$\mu_i$ is the Lebesgue measure for~$i\in\{1,2\}$.

\paragraph{Rigidity of kernels.}

We now turn our attention to rigidity results for kernels. Unless
otherwise specified, all topological hypothesis are to be understood
in~${\R}^2$ endowed with the usual topology; this is in particular the
case when we consider a measure with ``non-empty interior''.
We start with the following.

\begin{theorem}\label{thm:dens-dist}
  If~$(J_1,\mu_1,W_1)$ and~$(J_2,\mu_2,W_2)$ are two kernels such that
    $t(H,W_1)=t(H,W_2)$ for every function~$H\in\mathcal{H}$, then $\dist(W_1,W_2)=0$.
\end{theorem}

Theorem~\ref{thm:dens-dist} makes no assumption on the measure, but the
equality of~$W_1$ and~$W_2$ is only in measure, up to mapping to a common
probability space. Stronger rigidity properties are possible under some
regularity assumption on the kernels.  Define two points~$x$ and~$x'$ in~$J$ to
be \emph{twins for~$(J,\mu,W)$} if $W(x,y,z)=W(x',y,z)$ for $\mu^2$-almost
every pair~$(y,z)\in J^2$. The kernel~$(J,\mu,W)$ is \emph{twin-free} if it
admits no twins.

\begin{proposition}\label{prop:big}
  Let~$(J_1,\mu_1,W_1)$ and~$(J_2,\mu_2,W_2)$ be two twin-free
    kernels.  If $t(H,W_1)=t(H,W_2)$ for every function~$H\in\mathcal{H}$, then there
  exist two sets $N_1\subseteq J_1$ and~$N_2\subseteq J_2$ with
  $\mu_1(N_1)=0=\mu_2(N_2)$ and a bijection~$\rho\colon
  J_1\setminus N_1\to J_2\setminus N_2$ such that~$\rho$ and~$\rho^{-1}$
  are measure preserving, and~$W_1$ and~$W_2^{\rho}$ are equal
  $\mu_1^3$-almost everywhere.
\end{proposition}

For even more regular kernels, there are in fact distances on~$J_1$ and~$J_2$
such that a measure-preserving isometry exists. Specifically, the
\emph{neighborhood pseudo-distance} of a kernel~$(J,\mu,W)$ is the
function~$d_W\colon J \times J\to\mathbf{R}^+$ given by

\begin{equation}\label{eq:dW}
  \begin{aligned}
    \forall (x,x')\in J^2,\qquad d_W(x,x') &= \norm{W(x,\cdot,\cdot) -
      W(x',\cdot,\cdot)}_1\\ &=
    \int_{J^2}\abs{W(x,y,z)-W(x',y,z)}\diff\mu(y)\diff\mu(z).
  \end{aligned}
\end{equation}

\noindent
It is straightforward to check that~$d_W$ is reflexive and satisfies the
triangular inequality, so it is a distance if and only if it is separated,
\emph{i.e.}, if $(J,W,\mu)$ is twin-free.

A kernel~$(J,\mu,W)$ is \emph{pure} if (i) the map~$d_W$ is a distance on~$J$,
(ii) the metric space~$(J,d_W)$ is complete and separable, and (iii) the
measure~$\mu$ has \emph{full support}, that is, every non-empty open set
of~$(J,d_W)$ has positive $\mu$-measure.

\begin{theorem}\label{cor:pure}
  Let~$(J_1,\mu_1,W_1)$ and~$(J_2,\mu_2,W_2)$ be two pure kernels. If
    $t(H,W_1)=t(H,W_2)$ for every function~$H\in\mathcal{H}$, then there exists an
  isometry $\rho \colon(J_1, d_{W_1}) \to(J_2, d_{W_2})$ such that~$\rho$
    and~$\rho^{-1}$ are measure preserving, and~$W_1$ and~$W_2^{\rho}$ are
    equal~$\mu_1^3$-almost everywhere.
\end{theorem}

\subsubsection{Kernels based on chirotopes}\label{s:dWgeometric}

We now deduce Theorem~\ref{thm:proj} from Theorem~\ref{cor:pure} by studying
kernels coming from chirotopes. We let~$\chi$ be the map that sends a triple of
points in~${\R}^2$ to their orientation. We consider kernels of the form $(\supp
\mu, \mu, \chi_{|\supp \mu})$, and let~$d_{\chi,\mu}$ be the associated
neighborhood pseudo distance.\footnote{We depart here from the notation
introduced in Equation~\eqref{eq:dW} because we manipulate several kernels
based on the function~$\chi$, so we need to lift the ambiguity.}

\paragraph{Purity.}

Given two distinct points~$a$ and~$b$ in~${\mathbf{R}}^2$, we
let~$(ab)$ be the line through~$a$ and~$b$.

\begin{lemma}\label{lem:notwin}
    Let~$\mu$ be a measure on~${\R}^2$.  If~$\supp\mu$ has non-empty
  interior, then the kernel~$(\supp\mu,\mu,\chi)$ is twin-free.
\end{lemma}
\begin{proof}
   Let~$B$ be an open ball contained in~$\supp \mu$, and consider two
   distinct points~$a$ and~$b$ in~$\supp \mu$. There exists a line
   separating~$a$ from~$b$ and intersecting~$B$ (for otherwise~$B$
   would have to be contained in the line~$(ab)$).  Let~$x$ and~$y$ be
   two distinct points in~$(ab)\cap B$. By continuity of the
   determinant expressing orientations, there exists a positive
   real~$\epsilon$ such that any line secant to both~$B(x,\epsilon)$
   and~$B(y,\epsilon)$ also separates~$a$ from~$b$.  This implies that
   for any~$x' \in B(x,\epsilon)$ and~$y' \in B(y,\epsilon)$, we have
   $\chi(x',y',a) \neq \chi(x',y',b)$. Since $\mu^2\pth{B(x,\epsilon)
     \times B(y,\epsilon)}$ is positive, it follows that~$a$ and~$b$
   are not twins.
\end{proof}

\paragraph{Topologies.}

By Lemma~\ref{lem:notwin}, if the support of a measure~$\mu$
  has non-empty interior, then the
  kernel~$(\supp\mu,\mu,\chi)$ is twin-free and~$d_{\chi,\mu}$ is a
  distance. To show that this kernel is pure, it remains to control
  the topology induced by~$d_{\chi,\mu}$ on~$\supp \mu$.

\begin{lemma}\label{lem:dw:eq}
    Let~$\mu$ be a compactly supported measure that charges no line.
  If~$\supp\mu$ has non-empty interior, then the topologies induced
  on~$\supp\mu$ by~$d_{\chi,\mu}$ and~$\|.\|_2$ are equivalent.
\end{lemma}
\begin{proof}
    We first remark that~$d_{\chi,\mu}$ is a distance thanks to
    Lemma~\ref{lem:notwin}. We next argue that~$d_{\chi,\mu}$ is continuous
    toward~$\|.\|_2$. To do so, we consider a sequence $(x_n)$ of points
  in $\R^2$ converging to some point $x$ in the sense that $\lim_{n\to
    \infty} \|x_n-x\|_2 = 0$, and we prove that we also have $\lim_{n
    \to \infty} d_{\chi,\mu}(x_n,x)= 0$. For every pair~$(y,z) \in
  \supp\mu^2$ such that~$x$, $y$ and~$z$ are not aligned, the set of
  points~$x' \in \mathbf{R}^2$ satisfying the equation
    $\chi(x,y,z)=\chi(x',y,z)$ is an open half-space
  containing~$x$.  Consequently,~$\chi(x_n,y,z)$ is eventually
  equal to~$\chi(x,y,z)$.  It follows that the sequence $f_n(y,z)
  = \chi(x,y,z)-\chi(x_n,y,z)$ tends to~$0$ for $\mu^2$-almost
  every pair~$(y,z) \in \supp\mu^2$.  Indeed, this last statement is
  true if~$x$, $y$ and~$z$ are not aligned and the set of
  pairs~$(y,z)$ aligned to~$x$ is a $\mu^2$-nullset because~$\mu$ does
  not charge lines.  We conclude by an application of the dominated
  convergence theorem that the sequence $d_{\chi,\mu}(x_n,x)
  =\int_{\supp\mu^2} f_n(y,z) \diff\mu(y,z)$ tends to~$0$.

   To conclude, it suffices to prove that if $d_{\chi,\mu}(x_n,x)$ tends
   to~$0$ for some sequence ${(x_n)}_{n\in\N}\in\supp\mu^{\N}$ and
   $x\in\supp\mu$ then $\|x_n-x\|_2$ tends to~$0$ when~$n$ goes to
   infinity. We use a classical compactness argument.  Assume for the
   sake of contradiction that there is a sequence
   ${(d_{\chi,\mu}(x_n,x))}_{n\in\N}$ that tends to~$0$ while
   ${(\|x_n-x\|_2)}_{n\in\N}$ does not.  Since~$\supp\mu$ is compact,
   it is possible to extract a subsequence~${(x_{\phi(n)})}_{n\in\N}$
   that converges with respect to~$\|.\|_2$ to a limit~$y\in\supp\mu$
   different from~$x$. Then, as seen above,
   $d_{\chi,\mu}(x_{\phi(n)},y)\to 0$ and~$d_{\chi,\mu}(x,y)=0$.
Since~$d_{\chi,\mu}$ is a distance, it follows that~$x=y$, which yields a
contradiction.
\end{proof}

\begin{coro}\label{cor:chi-pur}
    If~$\mu$ is a measure that charges no line and such
    that~$\supp\mu$ is compact and has non-empty interior in~${\R}^2$, then the
  kernel~$(\supp\mu,\mu,\chi)$ is pure.
\end{coro}
\begin{proof}
By Lemma~\ref{lem:notwin}, we already know that~$(\supp\mu,\mu,\chi)$ is
    twin-free, and therefore that~$d_{\chi,\mu}$ is a distance. By
    Lemma~\ref{lem:dw:eq}, this distance defines the same topology on~$\supp
    \mu$ as the Euclidean distance, so~$(\supp \mu,d_{\chi,\mu})$ is complete
    and separable. Moreover, any open ball for~$d_{\chi,\mu}$ contains an open
ball for~$\|\cdot \|_2$ so~$\mu$ has full support.
\end{proof}

\paragraph{Alignments.}

To study how the map $\rho \colon\supp \mu_1 \to \supp \mu_2$ provided by
Theorem~\ref{cor:pure} transports alignments, we reformulate alignments in
topological terms.  Let us define the sets
\[ T^+ = \{(a,b,c) \in {(\R^2)}^3 \colon \chi(a,b,c)=1 \} \quad \text{and}
\quad T^- = \{(a,b,c) \in {(\R^2)}^3 \colon \chi(a,b,c)=-1 \}.\]
We define~$\cll(X)$ and~$\bll X$ to be the topological closure and the
boundary of a set~$X$ for the topology induced by the Euclidean distance,
respectively. From the expression of the orientation of a triple of points as a
determinant, it comes
  \begin{equation}\label{eq:aligned} \begin{aligned} \{(a,b,c) \in {(\R^2)}^3
  \colon \chi(a,b,c)=0 \} & = \bll T^+ = \bll T^-\\ & = \cll(T^+)\setminus T^+
  = \cll(T^-)\setminus T^- = \cll(T^+) \cap \cll(T^-).  \end{aligned}
  \end{equation}
We use a local version of this characterization of alignment, where we
restrict~$T^+$ or~$T^-$ to the support of a measure~$\mu$. Note that ${(\supp
\mu)}^3$ may contain isolated aligned triples or other pathologies, so we have
to take some care; we again rely on the assumption that $\supp \mu$ has
non-empty interior.

\begin{lemma}\label{lem:local-aligned}
  Let~$X$ be a subset of~$\R^2$ that contains an open ball~$B$. A
  triple $t=(x,y,z)\in X^3$ with~$x \in B$ is aligned if and only if
  it is in~$\cll(T^+ \cap X^3) \cap \cll(T^- \cap X^3)$.
\end{lemma}
\begin{proof}
  If~$t$ is aligned, then there exist~$x'$ and~$x''$ both in~$B$
  arbitrarily close to~$x$, and separated by the
  line~$(yz)$. It follow that exactly one of $(x',y,z)$ and
  $(x'',y,z)$ belongs to $T^+ \cap X^3$, and the other belongs to $T^-
  \cap X^3$. This implies that $t \in \cll(T^+\cap X^3) \cap
  \cll(T^-\cap X^3)$. The other direction follows from the fact that
    \[ \cll(T^+ \cap X^3) \cap \cll(T^- \cap X^3) \subseteq \cll(T^+) \cap \cll(T^-)\]
   and Equation~\eqref{eq:aligned}.
\end{proof}

\paragraph{Proof of Theorem~\ref{thm:proj}.} 

We can now prove that~if $\mu_1$ and~$\mu_2$ are two compactly
supported measures of~$\mathbf{R}^2$ that charge no line, whose
supports have non-empty interiors, and such that
$\ell_{\mu_1}=\ell_{\mu_2}$, then there exists a projective
transformation~$f$ such that $\mu_2\circ f=\mu_1$.

\begin{proof}[Proof of  Theorem~\ref{thm:proj}]
By Corollary~\ref{cor:chi-pur}, the kernels~$(\supp\mu_i,\mu_i,\chi)$ are pure.
    Thus, by Theorem~\ref{cor:pure}, there exists a measure-preserving
    isometry~$\rho \colon(\supp \mu_1, d_{\chi,\mu_1}) \to(\supp \mu_2,
    d_{\chi,\mu_2})$ such that~$\chi$ and~$\chi^{\rho}$ are equal
    $\mu_1^3$-almost everywhere. Moreover, Lemma~\ref{lem:dw:eq} ensures
    that~$\rho$ is an homeomorphism from~$(\supp \mu_1,\|\cdot\|_2)$ to~$(\supp
    \mu_2,\|\cdot\|_2)$. It remains to analyze how~$\rho$ transports
    alignments.

First, let us remark that~$\rho^3(T^+ \cap \supp \mu_1^3)$ and~$T^-\cap \supp
    \mu_2^3$ are disjoint. To prove this, we note that the intersection of
    these two sets is open in~$\supp \mu_2^3$ and has $\mu_2^3$-measure~$0$:
    indeed,
    we see that~$\chi$ and~$\chi^{\rho}$ disagree on every element of~$(T^+\cap \supp
    \mu_1^3) \cap {(\rho^{-1})}^3(T^- \cap\supp \mu_2^3)$, and
    therefore
    \[\mu_1^3((T^+ \cap\supp \mu_1^3) \cap {(\rho^{-1})}^3(T^- \cap\supp \mu_2^3)) = 0,\]
    which implies that~$\mu_2^3( \rho^3(T^+ \cap\supp \mu_1^3)
    \cap (T^- \cap\supp \mu_2^3) ) = 0$.  The conclusion follows
    because~$\mu_2^3$ has full support on~$\supp\mu_2^3$. A symmetric argument
    yields that~$\rho^3(T^- \cap \supp \mu_1^3)$ and~$T^+\cap \supp \mu_2^3$
    are also disjoint.

Now, consider an aligned triple~$t= (x,y,z)\in\supp\mu_1^3$ where~$x$ belongs
    to some open ball~$B_1$ contained in~$\supp \mu_1$. From
    Lemma~\ref{lem:local-aligned}, it comes that $t\in\cll (T^- \cap \supp
    \mu_1^3)\cap\cll(T^+ \cap \supp \mu_1^3)$. Thus, $\rho(t) \in \cll
    (\rho^3(T^-) \cap \supp \mu_2^3)\cap\cll(\rho^3(T^+) \cap \supp \mu_2^3)$.
    The disjointedness properties established in the previous paragraph imply
    that
  \[ \rho^3\pth{T^- \cap \supp \mu_{1}^3} \subseteq \R^2 \setminus T^+ =
    \cll \pth{T^-} \quad \hbox{and} \quad \rho^3\pth{T^+ \cap \supp
    \mu_{1}^3} \subseteq \R^2 \setminus T^- = \cll \pth{T^+}, \]
  so~$\rho(t) \in \cll\pth{T^-} \cap \cll\pth{T^+}$. Since~$\rho(x)$ belongs
    to~$\rho(B_1)$, which is an open ball contained in~$\supp \mu_2$,
    Lemma~\ref{lem:local-aligned} ensures that~$\rho(t)$ is aligned.

So~$\rho$ preserves alignments on every
    triple~$(x,y,z)\in\supp\mu_1^3$ with~$x\in B_1$. Recall that~$B_2
  = \rho(B_1)$ is an open, convex set contained in~$\supp\mu_2$. We
  show that~$\rho$ preserves alignments in general by a density
  argument. Let~$Q=abcd$ be a convex quadrilateral contained
  in~$B_1$ and such that $(ab) \cap (cd)$ is a unique

  \smallskip\noindent
  \begin{minipage}{9cm}
    point $e$, contained in~$B_1$, and~$(ad) \cap (bc)$ is a
      unique point~$f$, contained in~$B_1$. We now subdivide~$Q$, as
      depicted on the right, by defining~$x = (ac)\cap(bd)$, then
      introducing the intersections of~$(ex)$ with~$(ad)$ and~$(bc)$,
      and the intersections of~$(fx)$ with~$(ab)$ and~$(cd)$. This
      produces four new convex quadrilaterals tiling~$Q$. Remark that
      for any of these new quadrilaterals, the lines supporting
      opposite sides also intersect in~$e$ and~$f$. We apply this
      subdivision procedure recursively to any new quadrilateral
      produced, and let~$A$ be the (infinite) set of vertices of
      all the quadrilaterals obtained this way. The set~$A$ is dense
      in our initial quadrilateral~$Q$; indeed, the projective
    transformation that maps~$a$ to~$(0,0)$, $b$ to~$(1,0)$, $c$
    to~$(1,1)$ and~$d$ to~$(0,1)$ maps $A$ to the grid
    $\sst{(\frac{i}{2^k},\frac{j}{2^k})}{\text{$k\in\mathbf{N}$ and
        $0\le i,j\le2^k$}}$. (Note that under this transform, $e$
    and~$f$ are mapped to the line at infinity.)
  \end{minipage}
  \hfill
  \begin{minipage}{6cm}
    \begin{center}
      \begin{tikzpicture}[scale=1.3]
        \tikzstyle{node}=[circle,draw,fill=black,scale=0.4]
        \tikzstyle{newnode}=[node,fill=blue,color=blue]
        \begin{scope}[yscale=.4,xscale=.4]
          \node (e) at (-7,-1) [node,label=180:$e$] {};
          \node (f) at (1,-5) [node,label=0:$f$] {};
          \node (a) at (3,3) [node,label=45:$a$] {};
          \node (b) at (-2,1) [node,label=135:$b$] {};
          \node (c) at (-1,-1) [node,label=-100:$c$] {};
          \node (d) at (2,-1) [node,label=-45:$d$] {};
          \node (x) at (0,0) [newnode,label=45:$x$] {};
          \node (x1) at (2.333,.333) [newnode]{};
          \node (x2) at (.2,-1) [newnode]{};
          \node (x3) at (-1.4,-.2) [newnode]{};
          \node (x4) at (-.333,1.666) [newnode]{};
          \begin{scope}[on background layer]
            \draw[thick] (a) -- (b) -- (c) -- (d) -- (a);
            \draw[dashed] (c) -- (e) -- (b);
            \draw[dashed] (c) -- (f) -- (d);
            \draw (x4) -- (x2);
            \draw[dashed] (f) -- (x2);
            \draw[dashed] (x3) -- (e);
            \draw (x3) -- (x1);
          \end{scope}
        \end{scope}
      \end{tikzpicture}
    \end{center}
  \end{minipage}

  Now, observe that~$(\rho(a)\rho(b))$ and~$(\rho(c)\rho(d))$
  intersect in a unique point,
  namely~$\rho(e)\in\supp\mu_2$
  (note that~$\rho(\supp\mu_1)$
  cannot be contained in a line because~$\rho$ is measure preserving and~$\mu_2$
  does not charge lines). Similarly,~$(\rho(a)\rho(d))$
  and~$(\rho(b)\rho(c))$ intersect in a unique point, which
  is~$\rho(f)\in\supp\mu_2$.  We perform a similar recursive
  construction in~$B_2$, starting from the
  quadrilateral~$\rho(a)\rho(b)\rho(c)\rho(d)$, and let~$A'$ be
  the set of vertices obtained.

  Let~$p$ be the unique projective transformation that maps~$a$
    to~$\rho(a)$, $b$ to~$\rho(b)$, $c$ to~$\rho(c)$ and~$d$
    to~$\rho(d)$. In particular,~$p$ preserves the colinearity of
  points in~$\mathbf{R}^2$.  This implies that if $x\in A$, then
  $\rho(x)=p(x)$ as all points in~$A$ and in~$A'$ are defined from~$Q$
  and from~$(\rho(a),\rho(b),\rho(c),\rho(d))=(p(a),p(b),p(c),p(d))$
  using only colinearity properties.  Hence $p_{|Q}=\rho_{|Q}$ by
  continuity.

  It remains to show that $p(y)=\rho(y)$ for every~$y\in\supp\mu_1$.
  Let~$y\in\supp\mu_1$ and consider two lines~$L_1$ and~$L_2$ that
  both intersect the interior of~$Q$ and such that $L_1\cap
  L_2=\{y\}$. For~$i\in\{1,2\}$, let~$x_i$ and~$z_i$ be two distinct
  points in~$L_i\cap Q$; in particular $p(x_i)=\rho(x_i)$
  and~$p(z_i)=\rho(z_i)$. By definition,~$p$ preserves colinearity and
  since $x_1, x_2 \in B_1$, both triples~$\rho^3(x_i,y,z_i)$ are
  also aligned. It follows that
\[ \rho(y) = (\rho(x_1)\rho(z_1)) \cap (\rho(x_2)\rho(z_2)) = (p(x_1)p(z_1)) \cap (p(x_2)p(z_2)) = p(y),\]
which completes the proof.
\end{proof}

\subsubsection{Rigidity for kernels (proofs)}\label{sec:proof-kernels}

We now prove the rigidity results stated in
Section~\ref{s:overviewrigid}.

\paragraph{Step functions.}

We first argue that any kernel can be approximated by a step function.

\begin{lemma}\label{lem:stepfunction}
Let~$(J,\mu,W)$ be a kernel. For every
positive real number~$\varepsilon$, there exists a step function~$V$ on~$(J,\mu)$ such
that~$\norm{W-V}_1 \leq \varepsilon$.
\end{lemma}

\begin{proof}
It is well known in measure theory that for every function~$W$ and every
  positive~$\varepsilon$, there exists a function~$V\colon J^3\to
    \mathbf{R}$ that is measurable, has finite image (\emph{i.e.},~$V(J^3)$ is
  finite) and satisfies $\norm{W-V}_\infty\leq\varepsilon/2$. Let us
  write $V=\sum_{i=1}^ka_i\mathbf{1}_{A_i}$, where~$a_i$ is a nonzero real
  number and $\mathbf{1}_{A_i}$ the indicator function of a measurable
  set~$A_i$ for each~$i\in[k]$.

  Let~$\mathcal{B}$ be the set of boxes of the form~$P_1\times
  P_2\times P_3$, where~$P_1$, $P_2$ and~$P_3$ are measurable subsets
  of~$J$ and let~$X$ be the set of finite unions of elements
  of~$\mathcal{B}$. For every~$i\in[k]$, there is a set~$B_i \in X$
  such that $\mu(A_i\triangle B_i)\leq \frac{\varepsilon}{2 k
    \abs{a_i}}$. (A proof of this fact is presented in Appendix, \cref{lem:basic}.) The
  function~$U=\sum_{i=1}^ka_i\mathbf{1}_{B_i}$ is a step function as a
  linear combination of step functions.  Moreover,
  \[\norm{U-V}_1\leq \sum_{i=1}^k\abs{a_i}\cdot\norm{\mathbf{1}_{A_i}-\mathbf{1}_{B_i}}_1
  = \sum_{i=1}^k\abs{a_i}\cdot \mu(A_i \triangle B_i)\leq \frac{\varepsilon}{2}.\]
  Consequently, $\norm{W-U}_1\leq\norm{W-V}_1+\norm{V-U}_1\leq\varepsilon$,
  which finishes the proof.
\end{proof}

\paragraph{Induced step functions.}

Let~$\mu_n^u$ be the counting measure on~$[n]$. As reported earlier (in Subsection~\ref{s:overviewrigid}),
we know that~$([n],\mu_n^u,\mathbf{H}(W,S))$ is a step function for the partition
of~$[n]$ into singletons. We now bound the distance between this step
function and~$(J,\mu,W)$. Thanks to \cref{lem:stepfunction},
it is sufficient to deal with the case where~$(J,\mu,W)$ is also a step function.
We use the following (standard) terminology.
An \emph{atom} of a measure~$\mu$ is a $\mu$-measurable set~$A$
with $\mu(A)>0$ and such that every $\mu$-measurable subset of~$A$
has measure either~$\mu(A)$ or~$0$.
A measure is \emph{atom-free} if it admits no atom.

\begin{lemma}\label{lem:dWHWS}
Let~$(J,\mu,W)$ be a step function and let~$(P_1,\dotsc,P_m)$ be a partition
    of~$J$ such that~$W$ is constant on~$P_i\times P_j\times P_k$ for
    any~$(i,j,k)\in{[m]}^3$. For every integer~$n$ and every
    tuple~$S=(x_1,\dotsc,x_n) \in J^n$,
    \[ \dist(W,\mathbf{H}(W,S)) \leq 6\sum_{i=1}^m\abs{\mu(P_i)-\frac{\abs{\{j \in
    [n] \colon x_j \in P_i\}}}{n}}\cdot \norm{W}_\infty.  \]
\end{lemma}

\begin{proof}
  We first eliminate any atom the measure~$\mu$ may have by
  defining\footnote{If~$\mu$ is atom-free then we may take
    $(J_0,\mu_0)=(J,\mu)$ and~$\phi_1$ equal to the identity;
    otherwise, we set~$J_0=J\times[0,1]$, we let~$\mu_0$ be the
    product measure of~$\mu$ and the Lebesgue measure on~$[0,1]$, and
    we let~$\phi_1$ be the projection of~$J_0$ on its first coordinate.}
  an atom-free measurable space~$(J_0,\mu_0)$ and a measure-preserving
  map $\phi_1\colon J_0 \to J$. Now, let us write
    \[ Q_i = \{j \in [n] \colon x_j \in P_i\} \qquad \hbox{and} \qquad n_i = \abs{Q_i}.\]
  For~$i\in[m]$, we set~$a_i=\min\{\mu(P_i),\mu_n^u(Q_i)\}$.  Since
  $\mu_0(\phi^{-1}_1(P_i))=\mu(P_i)\geq a_i$, there is a subset~$R_i
    \subseteq \phi^{-1}_1(P_i)$ such that~$\mu_{0}(R_i)=a_i$.  The
  set~$J_0'=J_0 \setminus\bigcup_{i=1}^mR_i$ has
  measure~$1-\sum_{i=1}^ma_i=\sum_{i=1}^m(\mu_n^u(Q_i)-a_i)$. 
  A now classical theorem of measure theory, due to Sikorski~\cite{Sik58}\footnote{This result is often, and apparently wrongly, attributed to Sierpi\'nski.}, assures that we can
  partition~$J_0'$ into~$m$ measurable subsets~$T_1,\dotsc,T_m$ such
  that~$\mu_0(T_i)=\mu_n^u(Q_i)-a_i$ for every~$i\in[m]$. (A proof
  that such a partition exists is presented in Appendix, \cref{lem-subsetofmeasurex}.)

  We now construct a measure-preserving function~$\phi_2$ such that
  $\phi_2(R_i\cup T_i)=Q_i$ for each~$i\in[m]$. Recall
    that~$(Q_1,\dotsc,Q_m)$ is a partition of~$[n]$.  We know
  that~$\mu_0(R_i\cup T_i)=\mu_n^u(Q_i)=n_i/n$ for each~$i\in[m]$, so
    there is a partition~${(B_j^i)}_{j\in Q_i}$ of~$R_i\cup T_i$
  into~$n_i$ parts each of $\mu_0$-measure~$1/n$, which are indexed by
  the elements of~$Q_i$.  If $j\in Q_i$, then we set $\phi_2(y)=j$ for
  every~$y\in B_j^i$.  Doing this for each~$i \in[m]$ naturally
  defines a function~$\phi_2$ from~$J_0$ to~$[n]$.  (Indeed, for
  every~$x\in J_0$ there is a unique pair~$(i,j)\in[m]\times[n]$ such
  that $x\in B_j^i$, and we set~$\phi_2(x)=j$.)  For convenience, we
  note that for each~$j\in[n]$, there is exactly one index~$i\in[m]$
    such that $B_j^i$ is defined: we call~$i(j)$ this index.

  The function~$\phi_2$ satisfies that $\phi_2(R_i\cup T_i)=Q_i$
  for~$i\in[m]$.  Moreover, it is measure preserving since
    $\mu_0(\phi_2^{-1}(A))=\sum_{j\in A}\mu_0\left(B_j^{i(j)}\right)=\abs{A}/n=\mu_n^u(A)$
    for every~$A\subseteq[n]$.  To deduce our assertion, it suffices to
  prove that
  \[
      \norm{W^{\phi_1}-{\mathbf{H}(W,S)}^{\phi_2}}_1
  \leq
    6\sum_{i=1}^m\abs{\mu(P_i)-\mu_n^u(Q_i)}\cdot \norm{W}_\infty.
  \]
  If~$a$ belongs to~$R_i$ for some~$i\in[m]$, then by construction
  $\phi_1(a)\in P_i$ and moreover $\phi_2(a)\in Q_i$, that
    is,~$x_{\phi_2(a)}\in P_i$.  Consequently, for every~$(i,j,k)\in{[m]}^3$ the functions~$W^{\phi_1}$
    and~${\mathbf{H}(W,S)}^{\phi_2}$ are equal (and constant)
    on~$R_i\times R_j\times R_k$.

    It follows that~$W^{\phi_1}$ and~${\mathbf{H}(W,S)}^{\phi_2}$ are
  equal everywhere except on the set~$X=(J_0'\times J_0\times J_0)\cup(J_0\times
  J_0'\times J_0) \cup(J_0\times J_0\times J_0')$, which has
  $\mu_0^3$-measure at most
  $3\mu_0(J_0')=3\sum_{i=1}^m(\mu_n^u(Q_i)-a_i) \leq
  3\sum_{i=1}^m\abs{\mu(P_i)-\mu_n^u(Q_i)}$.
Consequently,
    \begin{align*}
  \norm{W^{\phi_1}-{\mathbf{H}(W,S)}^{\phi_2}}_1
        &\leq \int_X |W^{\phi_1}-{\mathbf{H}(W,S)}^{\phi_2}|\diff\mu_0^3\\
        &\leq \int_X|W^{\phi_1}|\diff\mu_0^3 + \int_X|{\mathbf{H}(W,S)}^{\phi_2}|\diff\mu_0^3\\
        &\leq 2\times\norm{W}_{\infty}\times\mu_0^3(X)\\
        &\leq 6\norm{W}_{\infty}\sum_{i=1}^m\abs{\mu(P_i)-\mu_n^u(Q_i)}.
    \end{align*}
    
    The statement follows.
\end{proof}

\paragraph{Random step functions.}

 To control how well~$\mathbf{H}(W,n)$ approximates~$W$, we use the
 following lemma.

\begin{lemma}\label{lem:sampling:norm}
  For every kernel~$(J,\mu,W)$ and every integer~$n$,
  \[
  \abs{\Ex\left[\norm{\mathbf{H}(W,n)}_1\right] - \norm{W}_1} \leq \frac{3}{n}\norm{W}_{\infty}.
  \] 
\end{lemma}
\begin{proof}
  Let~$S=(x_1,\dotsc,x_n)$ be a $\mu^n$-random tuple of~$J^n$.
  The expected value~$E$ of~$\norm{\mathbf{H}(W,S)}_1$ is equal to
  \[
  \int_{J^n}\frac{1}{n^3}\sum_{1\leq i,j,k\leq n}\abs{W(x_i,x_j,x_k)}\diff\mu^n(S)
  = \frac{1}{n^3}\sum_{1\leq i,j,k\leq n} I_{i,j,k},
  \]
  where~$I_{i,j,k}$ is the integral~$\int_{J^n}\abs{W(x_i,x_j,x_k)}\diff\mu^n(S)$.
    First note that $0 \leq I_{i,j,k} \leq \norm{W}_\infty$ for every~$(i,j,k)\in{[n]}^3$.
  If moreover~$i$, $j$ and~$k$ are pairwise different, then~$I_{i,j,k}=\norm{W}_1$.
    The number of triples of~${[n]}^3$ with pairwise different elements is greater than~$n^3-3n^2$.
  It follows that
  $(n^3-3n^2)\norm{W}_1\leq
  n^3E\leq (n^3-3n^2)\norm{W}_1+3n^2\norm{W}_\infty$,
  and further $-\frac{3}{n}\norm{W}_1\leq E-\norm{W}_1 \leq \frac{3}{n}\norm{W}_\infty$,
  which yields the conclusion since~$\norm{W}_1\leq\norm{W}_\infty$.
\end{proof}

We can now prove that random step functions are good approximations
relatively to our distance.

\begin{lemma}\label{lem:sampling}
  Let~$(J,\mu,W)$ be a kernel.  For every positive integer~$n$, we
  consider the random kernel~$([n],\mu_n^u,\mathbf{H}(W,n))$.  The
  random sequence~${(\dist(W,\mathbf{H}(W,n)))}_n$ almost surely tends
  to~$0$ as~$n$ goes to infinity.
\end{lemma}

\begin{proof}
  Let us first assume that~$W$ is a step function. We fix a
  partition~$(P_1,\dotsc,P_m)$ of~$J$ such that~$W$ is constant
  on~$P_i\times P_j\times P_k$ for any~$(i,j,k)\in{[m]}^3$ and apply
  \cref{lem:dWHWS} to obtain, for any $n$-tuple~$S$, the upper bound:
  \begin{equation}\label{eq:dWHWS}
      \dist(W,\mathbf{H}(W,S)) \leq 6\sum_{i=1}^m\abs{\mu(P_i)-\frac{\abs{\{j \in [n] \colon x_j \in P_i\}}}{n}}\cdot
  \norm{W}_\infty.
  \end{equation}
  Let us now consider a $\mu^n$-random tuple~$S$ and set~$Q_i =
    \{j \in [n] \colon x_j \in P_i\}$ and $n_i = \abs{Q_i}$. For
  each~$i\in[m]$, the parameter~$n_i$ follows a binomial law of
  parameter~$\mu(P_i)$.  Fix~$\varepsilon >0$. By Hoeffding's
  inequality, the probability that $\abs{\mu(P_i) - \frac{n_i}{n}}
  \geq \varepsilon$ is at most~$2e^{-2\varepsilon^2n}$.  Thus the
  union bound yields that $\dist(W,\mathbf{H}(W,S))
    \leq6m\varepsilon\norm{W}_\infty$ with probability at
  least~$1-2me^{-2\varepsilon^2n}$.  Hence,~$\dist(W,\mathbf{H}(W,n))$
  goes almost surely to~$0$ when~$n$ goes to infinity.

  Let us now consider the general case. By
  Lemma~\ref{lem:stepfunction}, there is a step function~$V$
  on~$(J,\mu)$ such that $\norm{W-V}_1\leq \varepsilon$.  For every
  positive integer~$n$, one can couple the random
  variables~$\mathbf{H}(V,n)$ and~$\mathbf{H}(W,n)$ such that
  \begin{equation}
    \Ex(\norm{\mathbf{H}(V,n)-\mathbf{H}(W,n)}_1)
    \leq\frac{3}{n}\norm{V-W}_\infty.
    \label{eq:correlate}
  \end{equation}
  Indeed, let~$S=(x_1,\dotsc,x_n)$ be a single $\mu^n$-random tuple
  of~$J^n$, and note that $\mathbf{H}(V,S)-\mathbf{H}(W,S) =
  \mathbf{H}(V-W,S)$.  Lemma~\ref{lem:sampling:norm} applied to~$V-W$
  yields Equation~\eqref{eq:correlate}. Since~$V$ is a step function,
  we know from the first part of the proof that when~$n$ goes to
    infinity, the random sequence~${(\dist(V,\mathbf{H}(V,n)))}_n$ 
  almost surely goes to~$0$. Let~$N$ be large enough to ensure that
    whenever~$n\geq N$, both $\Ex(\dist(V,\mathbf{H}(V,n)))\leq \varepsilon$ and
  $\frac{3}{n}\norm{V-W}_\infty\leq\varepsilon$ hold.
  It follows that for every~$n\geq N$,
  \begin{align*}
    \Ex(\dist(W,\mathbf{H}(W,n)))
    & \leq
    \Ex(\dist(W,V)+\dist(V,\mathbf{H}(V,n))
    +\dist(\mathbf{H}(V,n),\mathbf{H}(W,n)))\\
    & \leq
    \dist(W,V)+\Ex(\dist(V,\mathbf{H}(V,n)))
    +\Ex(\|\mathbf{H}(V,n)-\mathbf{H}(W,n)\|_1)\\
    & \leq
    \norm{W-V}_1 + \Ex(\dist(V,\mathbf{H}(V,n)))+\frac{3}{n}\norm{V-W}_\infty\\
    & \leq
    3\varepsilon.  
  \end{align*}
  The statement follows.
\end{proof}

\paragraph{Proof of Theorem~\ref{thm:dens-dist}.}

We can now prove the first rigidity result that we stated:
if~$(J_1,\mu_1,W_1)$ and~$(J_2,\mu_2,W_2)$ are two kernels such that
$t(H,W_1)=t(H,W_2)$ for every function~$H\in\mathcal{H}$, then $\dist(W_1,W_2)=0$.

\begin{proof}[Proof of Theorem~\ref{thm:dens-dist}]
  The assumption implies that for every positive integer~$n$, the two
  random variables~$\mathbf{H}(W_1,n)$ and~$\mathbf{H}(W_2,n)$ have
  the same distribution.  Consequently, if~$H_n$ is a random function
  equivalent to both~$\mathbf{H}(W_1,n)$ and~$\mathbf{H}(W_2,n)$, then
  by \cref{lem:sampling} both~$\dist(W_1,H_n)$
  and~$\dist(W_2,H_n)$ almost surely tend to~$0$.  The triangular
  inequality $\dist(W_1,W_2) \leq\dist(W_1,H_n)+\dist(H_n,W_2)$ thus
  ensures that~$\dist(W_1,W_2)$ equals~$0$.
\end{proof}

\paragraph{Kernel isomorphisms}

We next show that there is an isomorphism between (almost all) the
sets associated to these kernels that preserves (almost everywhere)
the characteristic of the kernels, that is, their measures and their
functions (\cref{prop:big}). We start with an analogue of a classical
result on kernels, which provides a weaker conclusion.

\begin{lemma}\label{lem:phi}
  If~$(J_1,\mu_1,W_1)$ and~$(J_2,\mu_2,W_2)$ are two kernels such that
  $\dist(W_1,W_2)=0$, then there exist a probability space~$(J,\mu)$
  and for each~$i\in\{1,2\}$ a measure-preserving map~$\phi_i\colon
  J\to J_i$ such that
  \begin{itemize}
  \item for every~$i\in\{1,2\}$ and every $\mu$-measurable set~$A\subseteq J$,
    the set $\phi_i(A)$ is~$\mu_i$-measurable; and
  \item $W_2^{\phi_2}(x,y,z)=W_1^{\phi_1}(x,y,z)$ for $\mu^3$-almost
    every~$(x,y,z)\in J^3$.
  \end{itemize}
\end{lemma}
\begin{proof}
  We only sketch the argument, as it is a straightforward extension of
  results already published, as referenced in what follows. We may
  first assume that $J_1=[0,1]=J_2$.  A proof of this fact for
  functions of two variable was given by Janson~\cite[Proof
    of~Theorem~7.1]{Jan13}.  His argument extends directly to
  functions of three variables and, for that matter, to any finite
  number of variables. (We point out that although Janson assume a
  kernel to be symmetric in its two variables, in the proof of
  Theorem~7.1 (and, more specifically, in the proof of Lemma~7.3),
  this assumption is used only to ensure that the obtained function is
  again symmetric (and thus a kernel in Janson's sense), which is not
  needed here.)

  Now a direct adaptation of the proof of Theorem~8.13 in the book by
  Lovász~\cite[p.~136]{Lov12} gives the lemma.  We just give an
  outline of the argument. We aim to show the statement of the theorem
  for~$J=J_1 \times J_2$, and~$\phi_i$ being the projection of~$J$
  on~$J_i$ for~$i\in\{1,2\}$.  We know that $\dist(W_1,W_2) =
  \inf_\mu\norm{W_1^{\phi_1}-W_2^{\phi_2}}^\mu_1$, where~$\mu$ ranges
  over all coupling measures of~$J=J_1 \times J_2$.  It suffices to
  show that this last infimum is in fact a minimum to deduce the
  statement.  As we assumed that~$J_1=[0,1]=J_2$, the space of
  coupling measures is compact in the weak topology.  Consequently, it
  is enough to show that the function~$\mu \mapsto \norm{W_1^{\phi_1}
    - W_2^{\phi_2}}^\mu_1$ is \emph{lower semicontinuous},
  \emph{i.e.}, if ${(\mu_n)}_n$ weakly converges to~$\mu$ then
\[
  \liminf_n \norm{W_1^{\phi_1} - W_2^{\phi_2}}^{\mu_n}_1 \geq
                              \norm{W_1^{\phi_1} - W_2^{\phi_2}}^\mu_1.
\]
This last inequality is Inequality~(8.21) on p.~137 of \emph{loc.~cit.} and the
proof follows as in the book.
\end{proof}

\paragraph{The twin-free case (proof of \cref{prop:big}).}

We can now prove a stronger rigidity: given two twin-free kernels
$(J_1,\mu_1,W_1)$ and~$(J_2,\mu_2,W_2)$, if $t(H,W_1)=t(H,W_2)$ for
every function~$H\in\mathcal{H}$, then there exist two sets~$N_1\subseteq J_1$
and~$N_2\subseteq J_2$ with~$\mu_1(N_1)=0=\mu_2(N_2)$ and an
invertible and measure-preserving map~$\rho\colon J_1\setminus N_1\to
J_2\setminus N_2$ such that
\begin{enumerate}
\item $\rho^{-1}$ is measure preserving; and\label{it:1}
\item $W_1$ and~$W_2^{\rho}$ are equal $\mu_1^3$-almost
  everywhere.\label{it:2}
\end{enumerate}

\begin{proof}[Proof of \cref{prop:big}]
  Theorem~\ref{thm:dens-dist} ensures that Lemma~\ref{lem:phi} applies:
  let~$(J,\mu)$ and~$\phi_1,\phi_2$ be the probability space and the
  applications given by this lemma, respectively. For~$i\in\{1,2\}$, we
  set~$\tilde{J}_i=\phi_i(J)$ and~$N_i=J_i\setminus\tilde{J}_i$. So
  $\mu_i(N_i)=0$ because $\phi_i$ is measure preserving and
  $\phi_i^{-1}(\tilde{J}_i)=J$.

  We start by showing that if two elements of~$J$ have the same image by~$\phi_1$, then
    they must have the same image by~$\phi_2$; that is, $\phi_2(\phi_1^{-1}(\{x\}))$ is a singleton
  for every~$x\in \tilde{J}_1$. Indeed, suppose that $\phi_1^{-1}(\{x\})$
  contains two distinct elements~$a$ and~$b$. Then
  $W_1^{\phi_1}(a,y,z)=W_1^{\phi_1}(b,y,z)$ for every~$(y,z)\in J^2$.
  Furthermore, we know that $W_1^{\phi_1}(a,y,z)=W_2^{\phi_2}(a,y,z)$ and
  $W_1^{\phi_1}(b,y,z)=W_2^{\phi_2}(b,y,z)$ for~$\mu^2$-almost every
  pair~$(y,z)\in J^2$.  Consequently,
  $W_2^{\phi_2}(a,y,z)=W_2^{\phi_2}(b,y,z)$ for~$\mu^2$-almost every
  pair~$(y,z)\in J^2$.  This implies that
  $W_2(\phi_2(a),y',z')=W_2(\phi_2(b),y',z')$ for~$\mu_2^2$-almost every
  pair~$(y',z')\in \tilde{J}_2^2$, because
  $\mu(J)=\mu(\phi_2^{-1}(\phi_2(J)))=\mu_2(\phi_2(J))$. (The last
  equality follows from the fact that~$\phi_2$ is measure preserving.) Since~$W_2$
  is twin-free, we deduce that $\phi_2(a)=\phi_2(b)$.  Therefore,
  we can define a map~$\rho\colon \tilde{J}_1\to \tilde{J}_2$ by
  setting~$\rho(x)$ to be the unique element of~$\phi_2(\phi_1^{-1}(\{x\}))$.

  Now, observe that there exists a subset~$N$ of~$\tilde{J}_1^3$ with
  $\mu_1^3(N)=0$ such that
\begin{itemize}
    \item $\tilde{J}_1^3\setminus N\subseteq{\phi_1(J)}^3$;
      \item $(x,y,z)\mapsto(\rho(x),\rho(y),\rho(z))$ is defined everywhere on $\tilde{J}_1^3\setminus N$; and
      \item $W_2^{\phi_2}(x,y,z)=W_1^{\phi_1}(x,y,z)$ whenever $(\phi_1(x),\phi_1(y),\phi_1(z))\notin N$.
\end{itemize}
For $(x_1,y_1,z_1)\in \tilde{J}_1^3\setminus N$, let $(x,y,z)\in J^3$ such that
$\phi_1(x)=x_1$, $\phi_1(y)=y_1$ and~$\phi_1(z)=z_1$. Then
$\rho(x_1)=\phi_2(x)$, $\rho(y_1)=\phi_2(y)$
and~$\rho(z_1)=\phi_2(z)$.  Therefore
$W_2^{\rho}(x_1,y_1,z_1)=W_2^{\phi_2}(x,y,z)=W_1^{\phi_1}(x,y,z)=W_1(x_1,y_1,z_1)$,
which yields~\ref{it:2}.

Let us show that~$\rho$ is measure preserving. We observe that if~$A\subseteq
\tilde{J}_2$, then $\rho^{-1}(A)=\phi_1(\phi_2^{-1}(A))$.  Indeed, by the
definitions
\begin{align*}
    \rho^{-1}(A)&=\sst{x\in \tilde{J}_1}{\phi_2(\phi_1^{-1}(\{x\}))\subseteq A}\\
            &=\sst{x\in \tilde{J}_1}{\phi_1^{-1}(\{x\})\subseteq\phi_2^{-1}(A)}\\
            &\subseteq \phi_1(\phi_2^{-1}(A)).
\end{align*}
Conversely, if $x\in\phi_1(\phi_2^{-1}(A))$, then there exists~$y\in J$
such that $\phi_1(y)=x$ and~$\phi_2(y)\in A$. Since
$\phi_2(\phi_1^{-1}(\{x\}))$ is a singleton, we deduce that
$\phi_2(\phi_1^{-1}(\{x\}))=\{\phi_2(y)\}$, which is contained in~$A$.
This proves the observation.
    
As a result, it suffices to prove that
$\mu_2(A)=\mu_1(\phi_1(\phi_2^{-1}(A)))$ to infer that
$\mu_1(\rho^{-1}(A))=\mu_2(A)$. As~$\phi_1$ and~$\phi_2$ are measure preserving,
it is enough to prove that $\phi_1^{-1}(\phi_1(\phi_2^{-1}(A)))=\phi_2^{-1}(A)$.
By definition, the set on the right side
is always contained in the set on the left side. For the
converse inclusion, fix $x\in\phi_1^{-1}(\phi_1(\phi_2^{-1}(A)))$ and
let us show that $\phi_2(x)\in A$.  There exists~$y\in J$ such that
$\phi_1(y)=\phi_1(x)$ and $\phi_2(y)\in A$. In particular,
$\phi_2(y)\in\phi_2(\phi_1^{-1}(\phi_1(\{x\})))$ and hence
$\phi_2(\phi_1^{-1}(\phi_1(\{x\})))=\{\phi_2(y)\}$. However,
$x\in\phi_1^{-1}(\phi_1(\{x\}))$ and thus $\phi_2(x)=\phi_2(y)\in A$.

To see that~$\rho$ is invertible, we first define $\rho'\colon \tilde{J}_2\to
\tilde{J}_1$.  To this end, one shows similarly as before that
$\phi_1(\phi_2^{-1}(\{x\}))$ is a singleton for every $x\in \tilde{J}_2$.
So $\rho'(x)$ can be defined as the unique element
of~$\phi_1(\phi_2^{-1})(\{x\})$. Now by symmetry of the roles played
by~$\phi_1$ and~$\phi_2$, one sees similarly as before that $\rho'$ is
measure preserving.  It remains to prove that $\rho'(\rho(x))=x$ for every~$x\in
\tilde{J}_1$. Fix~$x\in \tilde{J}_1$.  As
$\phi_2(\phi_1^{-1}(\{x\}))=\{\rho(x)\}$, there exists~$y\in J$ such that
$\phi_1(y)=x$ and~$\phi_2(y)=\rho(x)$. Consequently,
$x\in\phi_1(\phi_2^{-1}(\{\rho(x)\}))$, which is equal
to~$\{\rho'(\rho(x))\}$. This concludes the proof.
\end{proof}

\paragraph{The case of pure kernels (proof of \cref{cor:pure}).}

We finally establish our strongest rigidity theorem for kernels:
if~$(J_1,\mu_1,W_1)$ and~$(J_2,\mu_2,W_2)$ are two pure kernels and
$t(H,W_1)=t(H,W_2)$ for every function~$H\in\mathcal{H}$, then there exists an
isometry~$\rho \colon(J_1, d_{W_1}) \to(J_2, d_{W_2})$ such that~$W_1$
and~$W_2^{\rho}$ are equal $\mu_1^3$-almost everywhere.

\begin{proof}[Proof of \cref{cor:pure}]
  Let~$\rho \colon J_1\setminus N_1 \to J_2\setminus N_2$ be the map given by
  Proposition~\ref{prop:big} applied to~$(J_1,\mu_1,W_1)$
  and~$(J_2,\mu_2,W_2)$.  We first prove that we may restrict~$\rho$ to
  a set~$D_1$ with $\mu_1$-measure one such that if we fix any~$x \in D_1$, then
  $W_1(x,y,z)$ equals $W_2^\rho(x,y,z)$ for~$\mu_1^2$-almost every pair~$(y,z) \in
  J_1^2$. For~$x \in J_1$, we define~$I(x) \subset J_1^2$ to be the set of pairs~$(y,
  z)$ such that $W_1(x,y,z) \neq W_2^\rho(x,y,z)$. Further, let~$A$ be the set
  composed of each~$x \in J_1$ such that $\mu_1^2(I(x)) > 0$.

  We assert that $\mu_1(A)=0$. To prove this, we set~$A_\varepsilon = \sst{ x \in
  J_1}{\mu_1^2(I(x)) > \varepsilon }$ and we notice that $A = \bigcup_n
  A_{\varepsilon_n}$ where the union is taken over a decreasing
  sequence~${(\varepsilon_n)}_{n\in\mathbf{N}}$ that tends to~$0$. As the union is
  countable, it suffices to prove that $\mu_1(A_\varepsilon)=0$ for
  every~$\varepsilon > 0$ to conclude that $\mu_1(A)=0$. Fixing~$\varepsilon > 0$,
  it follows from the definitions that $W_1(x,y,z) \neq W_2^\rho(x,y,z)$ for every
  triple in~$\sst{(x,y,z)}{x \in A_\varepsilon\text{ and }(y,z) \in I(x)}$, which
  is a set of $\mu_1^3$-measure at least~$\varepsilon\cdot\mu_1(A_\varepsilon)$.
  Because of Property~\ref{it:2} of~Proposition~\ref{prop:big}, the previous
  statement implies that $\varepsilon\cdot\mu_1(A_\varepsilon) = 0$, hence
  $\mu_1(A_\varepsilon) = 0$.

  We define~$D_1 = J_1 \setminus N_1 \setminus A$ and~$D_2 = \rho(D_1) = J_2
  \setminus N_2 \setminus \rho(A)$. We know that $\mu_1(D_1) = 1$ and the
  equality $\mu_2(D_2) = 1$ follows from the fact that $\rho^{-1}$ is
  measure preserving.  The restriction~$\rho_{|D_1} \colon D_1 \to D_2$
  of~$\rho$ to~$D_1$ is an isometry between the metric spaces~$(D_1,d_{W_1})$
  and~$(D_2,d_{W_2})$.  Indeed, fixing~$(x,x') \in D_1^2$, we know from the
  construction of~$D_1$ that for~$\mu_1^2$-almost every pair~$(y,z) \in J_1^2$
  we have $W_1(x,y,z) = W_2^\rho(x,y,z)$ and $W_1(x',y,z) = W_2^\rho(x',y,z)$.
  So in particular $W_1(x,y,z) - W_1(x',y,z) = W_2^\rho(x,y,z)
  - W_2^\rho(x',y,z)$.  Consequently,
  \begin{align*}
    d_{W_1}(x,x')&=\int_{J_1^2}\abs{W_1(x,y,z) - W_1(x',y,z)} \diff\mu_1^2(y,z)\\
    &=\int_{J_1^2}\abs{W_2^\rho(x,y,z) - W_2^\rho(x',y,z)} \diff\mu_1^2(y,z)\\
    &=\int_{J_2^2}\abs{W_2(\rho(x),y',z') - W_2(\rho(x'),y',z')} \diff\mu_2^2(y',z')\\
    &=d_{W_2}(\rho(x),\rho(x')).
  \end{align*}
  This proves that $\rho_{|D_1}$ is an isometry.
  
  Now we assume that $\mu_1$ has full support and $(J_2,d_{W_2})$ is complete.
  In this case, $\rho_{|D_1}$ extends by continuity to an injective
  map~$\tilde{\rho}$ on~$J_1$. To prove this, it suffices to show that
  $\rho_{|D_1}$ is absolutely continuous and~$D_1$ is dense
  in~$(J_1,d_{W_1})$. The absolute continuity follows from the fact that
  $\rho_{|D_1}$ is an isometry. The set~$D_1$ is dense in~$J_1$ because
  every open set included in~$J_1 \setminus D_1$ is an open nullset, and hence
  is empty as~$\mu_1$ has full support.  By continuity of~$d_{W_1}$
  and~$d_{W_2}$ towards themselves the extension~$\tilde{\rho}$ is an isometry.

  To prove the second item, it suffices to apply the previous proof to the
    inverse~${(\rho_{|D_1})}^{-1}$ of~$\rho_{|D_1}$, where the roles played
  by~$(J_1,\mu_1,W_1)$ and~$(J_2,\mu_2,W_2)$ are inverted.
\end{proof}

\subsection{Flexibility of realizations and the probability that \texorpdfstring{$k$}{k} random points are in convex position}\label{s:nonrigid}

For limits of order types that can be realized by at least one measure
whose support has non-empty interior, the rigidity theorem completely
describes the space of realizations: they are the orbit of that one
realization under spherical transformations. In general, the situation
can be radically different as the following easy example shows.

\begin{example}
Let~$\ell_\diamond$ be the limit of a sequence of sets of points in
convex position. This limit assigns the probability~$1$ to each~$\diamond_k$, the order type of~$k$ points in convex position, and~$0$
to the rest. Every measure with convex support realizes~$\ell_\diamond$; this allows arbitrarily disconnected support, as
wells as support of any Hausdorff dimension between~$0$ and~$1$
(consider a Cantor set on~$[0,1]$ with that dimension and map the
interval to the circle).  
\end{example}

\noindent
The limit~$\ell_\diamond$ is exceptionally simple, it obviously maximizes~$\ell(\diamond_k)$. One may wonder
if this variety of realizations is also exceptional.
We construct a different limit with similar realization properties that plays a
role in combinatorial geometry as it gets close to minimizing~$\ell(\diamond_k)$ for~$k$ large enough.

\begin{theorem}\label{t:noregrep}
    There exists a limit~$\ell_E$ of order types such that
  for every~$t \in (0,1)$, the limit~$\ell_E$ can be
    realized by a measure with a support of Hausdorff dimension~$t$. Moreover, there is no measure~$\mu$ that realizes~$\ell_E$
    and is, on an open set of positive $\mu$-measure, absolutely
    continuous to the Lebesgue measure or to the length measure on a~$C^2$ curve of positive length.
\end{theorem}

\noindent
\Cref{t:noregrep} will follow from \cref{l:flat} for the first statement,
and \cref{l:decay-LE} and~\cref{l:sylvester} for the second
statement.

\begin{figure}[!t]
  \begin{center}
  \def\a{0.4} \def\b{0.30} \def\c{3.6} \def\d{3.6}
  \newcommand\phizero[1]{
    \begin{scope}[xscale=\a,yscale=\b] #1 \end{scope}}
  \newcommand\phione[1]{
    \begin{scope}[xshift=\c cm,yshift=\d cm]
      \begin{scope}[yscale=-\b,xscale=-\a]
        #1 \end{scope}\end{scope}}
  \newcommand{\bigphi}[1]{\phizero{#1}\phione{#1}}
  \newcommand{\mytikz}[1]{\begin{tikzpicture}#1
      \node at (0,-.6) {};\end{tikzpicture}}
  \def\rect{\draw[fill=gray,color=gray] (0,0) rectangle (\c,\d);}
  \def\myframe{\draw[dashed] (0,0) rectangle (\c,\d);}
  \begin{tabular}{c c c}
    \mytikz{ 
      \myframe%
      \bigphi{\rect%
        \draw[<->,>=stealth,thick] (0,0) -- ({\c},0) node[midway,above] {$a$};
        \draw[<->,>=stealth,thick] (0,0) -- (0,\d) node[midway,right] {$b$};}
      \draw[<->,>=stealth,thick] (0,-0.2) -- ({\c},-0.2) node[midway,below] {$1$};
      \draw[<->,>=stealth,thick] (-0.2,0) -- (-0.2,\d) node[midway,left] {$1$};
    }&
    \mytikz{\myframe\bigphi{\bigphi{\rect}}}&
    \mytikz{\myframe\bigphi{\bigphi{\bigphi{\rect}}}}\\
    $A_1$ & $A_2$ & $A_3$
  \end{tabular}
  \end{center}
  \caption{Definition of~$\ell_E$.\label{f:lE}}
\end{figure}
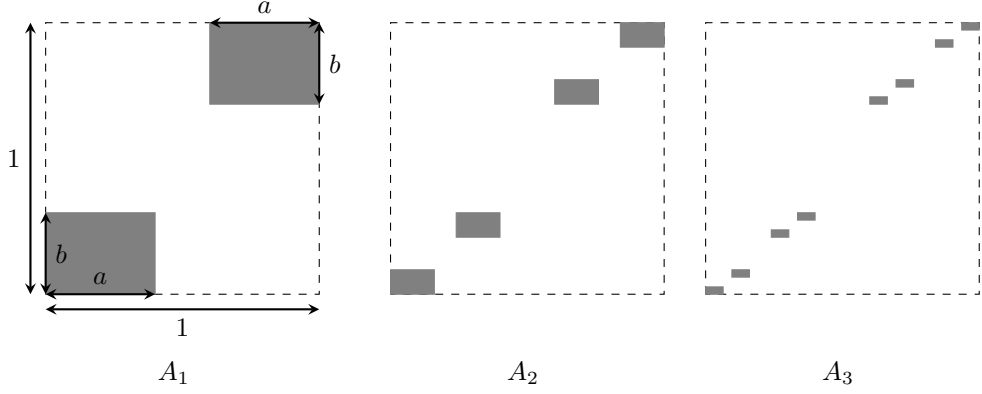

\subsubsection{Definition of~\texorpdfstring{$\ell_E$}{l_E}}

It is convenient to give two presentations of~$\ell_E$, one geometric
and the other combinatorial.

\medskip

Let us start with the combinatorial definition of~$\ell_E$. Consider
the space~$E = {\{0,1\}}^\N$ equipped with the coin-tossing measure. For~$u,v \in E$,
let~$u \wedge v$ be the longest common prefix of~$u$
and~$v$ and let~$\prec_{lex}$ be the lexicographic order on~$E$. We
define $\ell_E$ as a chirotope $\chi$ on $E$. Specifically, let~$u,v,w
\in E$ and, without loss of generality, suppose that $u \prec_{lex} v
\prec_{lex} w$. We set~$\chi(u,v,w) = 1$ if~$|u\wedge v| < |v \wedge
w|$ and~$\chi(u,v,w) = -1$ otherwise. For any order type~$\omega$ of
size~$k$, we let~$\ell_E(\omega)$ be the probability that the
restriction of~$\chi$ to~$k$ random elements of~$E$ chosen
independently from the coin-tossing distribution equals, after
unlabelling,~$\omega$. The fact that~$\ell_E$ is a limit of order types
easily follows from the geometric viewpoint.

\medskip

Let us now give a geometric presentation of~$\ell_E$; refer to
      \cref{f:lE}. As is usual, let~$\{0,1\}^*$ be the collection of all finite binary words.
      Fix some parameters~$a$ and~$b$ such that $0 < b < a < \frac{1}{2}$, and
    define the rectangles~$R = {[0,1]}^2, R_0 = [0,a] \times [0,b]$ and~$R_1
      = [1-a, 1] \times [1-b,1]$. For each~$i\in\{0,1\}$, let~$\varphi_i$ be the
affine transform fixing~$(i,i)$ and mapping~$R$ to~$R_i$.
To any word~$w = i_1i_2 \dotso i_n \in
{\{0,1\}}^*$ we associate the set~$R_{w} = \varphi_{i_n} \circ \varphi_{i_{n-1}}
      \circ \dots \circ \varphi_{i_1}(R)$ and let~$\mu^{a,b}$ be the probability
measure such that $\mu^{a,b}(R_w)=\frac{1}{2^{|w|}}$ for every~$w \in
{\{0,1\}}^*$. We notice that $R_w \subset R_v$ if and only if~$v$ is a
prefix of~$w$. Letting~$A_n = \bigcup_{w \in {\{0,1\}}^n} R_w$ for~$n
\geq 1$, the support of~$\mu^{a,b}$ is~$A = \bigcap_{n\ge1} A_n$.

\begin{lemma}\label{l:flat}
    If~$(a,b)\in{\left(0,\frac12\right)}^2$ and~$b \leq (1 - 2a)(1 - 2b)a$, then $\ell_{\mu^{a,b}} = \ell_E$.
      In particular,~$\ell_E$ is a limit of order types.
\end{lemma}
\begin{proof}
  The measure~${\mu^{a,b}}$ is the image of the coin-tossing probability
    on~${\{0,1\}}^\N$ by the function $\Psi_{a,b}$ that assigns to~$w \in
    {\{0,1\}}^\N$ the unique point in~$\bigcap_{w_v} R_{w_v}$, where the
  intersection is taken over all prefixes~$w_v$ of~$w$.

  We shall prove that every point in~$A \cap R_1$ lies above any line
  spanned by two points in~$A \cap R_0$ provided that
  \[ b \leq (1 - 2a)(1 - 2b)a.  \]
  Since~$A$ is stable by the symmetry of center~$(\frac12,\frac12)$,
  it would then follow that every point in~$A\cap R_0$ lies below any line
  spanned by two points in~$A\cap R_1$. Let us show how this property of~$A$
    allows us to conclude the proof.
  Indeed, this property yields that~$\ell_{{\mu^{a,b}}}$ is fully determined.
  Let~$\Psi_{a,b}(u)$, $\Psi_{a,b}(v)$ and~$\Psi_{a,b}(w)$ be three pairwise distinct points
    in~$A$ with~$u,v,w\in{\{0,1\}}^\N$
  and assume that $u\prec_{\text{lex}}v\prec_{\text{lex}}w$.
  If~$|u\wedge v|<|v\wedge w|$, set~$p=u\wedge v$. 
  Since~$u\prec_{\text{lex}}v$, the word~$p.0$ is a prefix of~$u$
  and~$p.1$ is a prefix of~$v$, and therefore of~$w$.
    It follows that~$\Psi_{a,b}(u)\in R_{p.0}$ and~$\{\Psi_{a,b}(v),\Psi_{a,b}(w)\}\subset R_{p.1}$.
  Moreover, the abscissa of~$\Psi_{a,b}(v)$ is smaller than that of~$\Psi_{a,b}(w)$
  because~$v\prec_{\text{lex}}w$.
  Consequently, $\chi(\Psi_{a,b}(u),\Psi_{a,b}(v),\Psi_{a,b}(w))=1=\chi_E(u,v,w)$.
  The proof that $\chi(\Psi_{a,b}(u),\Psi_{a,b}(v),\Psi_{a,b}(w))=\chi_E(u,v,w)$
  when~$|u\wedge v|>|v\wedge w|$ is similar and we omit it.
  
  It remains to prove that~$A$ indeed fulfills the property announced.
  For two distinct points~$x,y\in A$, let~$\alpha(x,y)$ be the angle between
  the line~$h(x,y)$ and the abscissa axis (so~$\alpha(x,y)$ is defined modulo~$\pi$).
  If~$x\in R_0$ and~$y\in R_1$, then
  \[
    1-2b \leq \tan\alpha(x,y) \leq \frac{1}{1-2a},
  \]
  since the minimum is obtained when~$x_m=(0,b)$ and~$y_m=(1,1-b)$, while
  the maximum is obtained when~$x_M=(a,0)$ and~$y_M=(1-a,1)$ (see Figure~\ref{fig-Risflat}, left.)
  The application of a function~$\varphi_i$ with~$i\in\{1,2\}$ acts as follows
    \[\tan\alpha(\varphi_i(x),\varphi_i(y))=\frac{b}{a}\tan\alpha(x,y).\]
  Since~$\frac{b}{a}<1$, it further holds
  that~$\tan\alpha(\varphi_i(x),\varphi_i(y))<\tan\alpha(x,y)$.
By iterating this property, it eventually follows that~$\tan\alpha(x,y)\leq\frac1{1-2a}$
  for every~$x,y\in A$ such that~$x\neq y$ and~$x$ has smaller abscissa than~$y$.

  Let~$x,y\in A\cap R_0$ and~$z\in A\cap R_1$ such that~$x$ has smaller abscissa than~$y$.
  Note that~$z$ lies above the line~$h(x,y)$ if and only if
  $\alpha(x,y)\leq\alpha(x,z)$, where the values of the angles are taken
  in~$(-\frac\pi2,\frac\pi2)$ (see Figure~\ref{fig-Risflat}, right),
  which in turn is equivalent to $\tan\alpha(x,y)\leq\tan\alpha(x,z)$.
  Since~$\tan\alpha(x,y)\leq \frac{b}{a}\frac{1}{1-2a}$
  and~$\tan\alpha(x,z)\geq 1-2b$, it suffices that
  $\frac{1}{1-2a}\frac{b}{a}\leq 1-2b$,
    \emph{i.e.},~$b\leq(1-2a)(1-2b)a$.
  This proves the announced property, thereby completing the demonstration.
\end{proof}

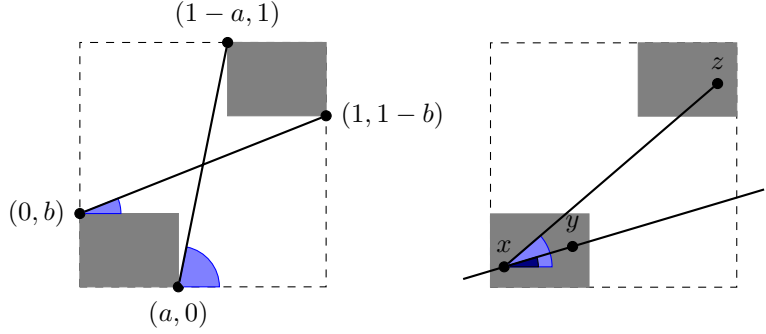
\begin{figure}
    \begin{center}
    \def\a{0.4} \def\b{0.3} \def\c{3.6} \def\d{3.6} \def\sc{.9}
    \newcommand\phizero[1]{\begin{scope}[xscale=\a,yscale=\b] #1 \end{scope}}
    \newcommand\phione[1]{
      \begin{scope}[xshift=\c cm,yshift=\d cm]
        \begin{scope}[yscale=-\b,xscale=-\a]
          #1 \end{scope}\end{scope}}
    \tikzstyle{node}=[circle,draw,fill=black,scale=0.4]
    \def\rect{\draw[fill,color=gray] (0,0) rectangle (\c,\d);}
    \def\myframe{\draw[dashed] (0,0) rectangle (\c,\d);}
    \begin{tikzpicture}[scale=\sc]
        \myframe
        \phizero{\rect
          \node[node,label=180:{$(0,b)$}] (A1) at (0,\d) {};
          \node[node,label=-90:{$(a,0)$}] (B1) at (\c,0) {};
        }
        \phione{
          \rect
          \node[node,label=0:{$(1,1-b)$}] (A2) at (0,\d) {};
          \node[node,label=90:{$(1-a,1)$}] (B2) at (\c,0) {};
        }
        \node at (0,-.6) {};
        \draw[blue,fill=blue!50] (A1)--++(0:.6) arc (0:22:.6)--(A1);
        \draw[blue,fill=blue!50] (B1)--++(0:.6) arc (0:78:.6)--(B1);
        \draw[thick](A1)--(A2)(B1)--(B2);
      \end{tikzpicture}
    \begin{tikzpicture}[scale=\sc]
        \myframe
        \phizero{\rect
          \node[node,label=90:$x$] (X) at (.5,1) {};
          \node[node,label=90:$y$] (Y) at (3,2) {};
          \coordinate (X') at (-1,.4);
          \coordinate (Y') at (10,4.8);
        }
        \phione{
          \rect
          \node[node,label=90:$z$] (Z) at (.7,2) {};
        }
        \node at (0,-.6) {};
        \draw[blue,fill=blue!50] (X)--++(0:.7) arc (0:41:.7)--(X);
        \draw[black,fill=blue!50!black] (X)--++(0:.5) arc (0:17:.5)--(X);
        \draw[thick](X)--(Z)(X')--(Y');
      \end{tikzpicture}
  \end{center}
  \caption{
    Left: Minimal and maximal value for~$\alpha(x,y)$ when~$x\in R_0$ and~$y\in R_1$.
    Right: $z\in R_1$ is above~$h(x,y)$ if~$\alpha(x,y)<\alpha(x,z)$.\label{fig-Risflat}}
\end{figure}

\subsubsection{The Erd\H{o}s-Sylvester problem} 

As mentioned already in the introduction, the limit of order types~$\ell_E$ was
    constructed to attack the Erd\H{o}s-Sylvester problem of determining~$c_k$.
    In fact, both upper and lower bounds on~$c_k$ go back to the classical
    results of Erd\H{o}s and Szekeres who proved that for every integer~$k$,
    there exists~$s(k)$ such that any planar point set in general position and
    of size at least~$s(k)$ contains~$k$ points in convex position. Erd\H{o}s
    and Szekeres also gave a construction showing that~$s(k)\geq2^{k-2}+1$ and
    they conjectured this value to be tight. This bound was nearly achieved in a recent
    breakthrough of Suk~\cite{Suk-ES} who improved the upper bound
    on~$s(k)$ to~$2^{k+6k^{2/3} \log k}$.  The next proposition is folklore.

\begin{proposition}\label{p:decay-low}
      For every integer~$k\ge 4$ and every limit of order types~$\ell$,
      \[\ell(\diamond_k) \ge 2^{-k^2+\littleo{k^2}}.\]
\end{proposition}
\begin{proof}
    Let~${(\ot_n)}_{n \in \N}$ be a convergent sequence of order types
  with limit $\ell$. The Erd\H{o}s-Szekeres theorem ensures that $p(\diamond_k,\ot) \ge
  \frac1{\binom{s(k)}k}$ for any order type~$\ot$ of size~$s(k)$.
    Let~$k \ge 4$ and fix some~$n_0$ such that
$|\ot_n| \ge s(k)$ whenever~$n \ge n_0$. It then follows that
\[ \forall n \ge n_0, \qquad p(\diamond_k,\ot_n) = \sum_{\ot \in
  \sot_{s(k)}} p(\diamond_k,\ot)p(\ot,\ot_n) \ge \frac1{\binom{s(k)}k}
\sum_{\ot \in \sot_{s(k)}}p(\ot,\ot_n) = \frac1{\binom{s(k)}k}\]
The first identity is a standard conditional probability argument:
instead of taking a random $k$-element subset of a realization of~$\ot_n$, we
      first take an $s(k)$-element subset, consider their order type~$\ot$, then
      take a random $k$-element subset of a realization of~$\ot$ and estimate the
      probability that it has order type~$\diamond_k$
conditioned on the order type of~$\ot$. The last identity simply
expresses that the sum for all order types $\ot$ of size $s(k)$ of the
density $p(\ot,\ot_n)$ is~$1$. We derive that $p(\diamond_k,\ot_n) \ge
    \frac1{\binom{s(k)}k} \ge \frac{1}{2^{k^2+\littleo{k^2}}}$ for any~$n \ge
n_0$, so $\ell(\diamond_k) = \lim_{n \to \infty} p(\diamond_k,\ot_n)
    \ge \frac{1}{2^{k^2+\littleo{k^2}}}$.
\end{proof}

\noindent
Our example~$\ell_E$ matches the order of growth in the exponent of
\cref{p:decay-low}.  We don't know of any previous result in this direction.

\begin{proposition}\label{l:decay-LE}
    $\ell_E(\diamond_k) = 2^{-\frac{k^2}8+\bigo{k \log k}}$.
\end{proposition}
\begin{proof}
  Let~$\mu$ be a measure that does not charge lines. Let~$X$ and~$S$
  be two sets of random points chosen independently from~$\mu$ of respective sizes~$k$ and~$\lfloor k/2 \rfloor$.
  By the discussion above,
    \[\ell_\mu(\diamond_k) = \Prob_\mu(\text{$X$ is in convex position})\leq \binom{k}{\lfloor \frac k 2\rfloor} \Prob_\mu(\text{$S$ is a cup or a cap}).\]
  Moreover, the construction of~$\mu^{a,b}$ implies that
  \[\Prob_{\mu^{a,b}}(\text{$S$ is a cup or a cap})=2\Prob_{\mu^{a,b}}(\text{$S$ is a cup}).\]

  The condition that~$b \leq (1 - 2a)(1 - 2b)a$ ensures that every
  $s$-cup containing more than one point in~$\phi_0(R)$ contains at
  most one point in~$\phi_1(R)$. It follows that
  \[\begin{aligned}
  \Prob_{\mu^{a,b}}(\text{$S$ is a cup}) = &
        \Prob_{\mu^{a,b}}(\text{$S$ is a cup and $S \subset \phi_0(R)$}) + \Prob_{\mu^{a,b}}(\text{$S$ is a cup and $S \subset \phi_1(R)$})\\
        & + \Prob_{\mu^{a,b}}(\text{$S \cap \phi_0(R)$ is a cup and $|S \cap \phi_0(R)| = |S|-1$})
  \end{aligned}\]
  Observe that
      \[ \Prob_{\mu^{a,b}}(\text{$S$ is a cup $|\, S \subset \phi_0(R)$}) = \Prob_{\mu^{a,b}}(\text{$S$ is a cup $|\,S \subset \phi_1(R)$}) = \Prob_{\mu^{a,b}}(\text{$S$ is a cup}).\] 
  Altogether, defining~$f(s)$ to be the probability that~$s$ random
  points chosen independently from~$\mu^{a,b}$ form a cup, we have
  \[ f(s) = \frac{2}{2^{s}} f(s) + \frac {s}{2^s} f(s-1), \quad \text{that is,} \quad  f(s) = \frac{s}{2^s-2}f(s-1).\]
  Notice that $f(3) = \frac12$. This can be seen directly from
  the combinatorial description of~$\ell_E$ as follows. Consider three different
    sequences~$u,v,w \in {\{0,1\}}^n$, assuming that $u \prec_{lex} v
  \prec_{lex} w$. The orientation of~$(u,v,w)$ depends only on
  the first entry of~$v$ where $u$ and~$w$ differ, and this entry of~$v$
  is uniformly distributed in~$\{0,1\}$. Altogether,
  \[ f(s) = \prod_{i=3}^s \frac{i}{2^i-2} = 2^{- \frac{s(s+1)}2 +3} \prod_{i=3}^s\frac{i}{1-2^{-i+1}}\]
    so $f(s) = 2^{-\frac{s^2}2 + \bigo{s \log s}}$ and~$\ell_E(\diamond_k) = 2^{-\frac{k^2}8 + \bigo{k \log k}}$.
\end{proof}

\noindent
Until recently it was suspected that every realization of the order types of
Erd\H{o}s-Szekeres needed to be very spread out, in the sense that the quotient of the
diameter and the minimal distance between two points on every realization had
to be exponentially large. A construction of the Erd\H{o}s-Szekeres example in
a grid of size polynomial in~$n$ was recently achieved~\cite{ERZ}.

To some extent our next result vindicates the original intuition on
realizations of the order types of Erd\H{o}s-Szekeres. We show that the fast
decay of~$\ell_E(\diamond_k)$ exhibited in \cref{l:decay-LE} drastically
restricts its space of realizations.

\begin{lemma}\label{l:sylvester}
  Let~$\mu$ be a finite measure over~$\R^2$ for which lines are
  negligible and~$U$ an open set of positive $\mu$-measure.
  \begin{itemize}
  \item[(i)] If~$\mu$ is absolutely continuous, on~$U$, to the
    Lebesgue measure then $p(\diamond_k,\mu) \ge 4^{-k \log k + \bigo{k}}$.
  \item[(ii)] If~$\mu$ is absolutely continuous, on~$U$, to the length
      measure on a~$C^2$ curve then $p(\diamond_k,\mu) \ge 2^{-\bigo{k}}$.
  \end{itemize}
\end{lemma}
\begin{proof}
    By the Radon-Nikodym theorem~\cite{Nik30,Rad13}, if a measure~$\mu$ is absolutely
  continuous to a measure~$\lambda$ on a measurable set~$X$ then there
  exists an absolutely continuous function~$\frac{d\mu}{d\lambda}$
  such that $\mu(A)=\int_A\frac{d\mu}{d\lambda} d\lambda$ for every
  measurable set~$A\subseteq X$.

  \medskip
  
  Let~$\lambda_2$ be the Lebesgue measure over~$\R^2$. Since
  $\mu(U) > 0$ and~$\mu$ is absolutely continuous to~$\lambda_2$ on~$U$,
    the function~$\frac{d\mu}{d\lambda_2}$ is nonnegative, nonzero,
  and continuous on~$U$. In particular,~$U$ contains some square~$A$
  on which~$\frac{d\mu}{d\lambda_2}$ is bounded from below by some
  positive constant~$c$. The probability that~$k$ random points chosen
  independently from~$\mu$ are all in~$A$ is
  $\pth{\frac{\mu(A)}{\mu(\R^2)}}^k$. Conditioning on this event, we deduce that
  \[ p(\diamond_k,\mu) \ge \pth{\frac{\mu(A)}{\mu(\R^2)}}^k p(\diamond_k, \mu_{|A}).\]
  Let~$1_{\diamond_k}$ be the indicator function,
  over~$A^k$, of~$k$-tuples of points in convex position. We have
  \[\begin{aligned}
  p(\diamond_k, \mu_{|A}) = & \int_{A^{k}}1_{\diamond_k}
  d\mu(x_1)d\mu(x_2) \dotso d\mu(x_k)\\
  = & \int_{A^{k}}1_{\diamond_k} \frac{d\mu}{d\lambda_2}(x_1)\frac{d\mu}{d\lambda_2}(x_2)\dotso \frac{d\mu}{d\lambda_2}(x_k)d\lambda_2(x_1)d\lambda_2(x_2)\dotso d\lambda_2(x_k)\\
  \ge & c^k  \int_{A^{k}}1_{\diamond_k} d\lambda_2(x_1)d\lambda_2(x_2)\dotso d\lambda_2(x_k) = c^k  p(\diamond_k, \lambda_{2|A})
 \end{aligned}\]
  Valtr~\cite{Valtr} proved that $p(\diamond_k,\lambda_{2|A}) =
  \frac1{k!^2}\binom{2k-1}{k-1}^2$, so altogether
    \[ p(\diamond_k,\mu) \ge \pth{\frac{\mu(A)}{\mu(\R^2)}}^k \frac{c^k}{k!^2}\binom{2k-1}{k-1}^2 = \Omega\pth{4^{-k \log k+\bigo{k}}},\]
  which proves the first statement.

\medskip
  
For the $1$-dimensional case, let~$d\lambda_1$ be the $1$-dimensional Lebesgue
    measure and let~$\Gamma$ be a~$C^2$ curve such that~$\mu$ is absolutely
    continuous to~$\lambda_{1|\Gamma}$ on some open set~$U$. Since
    $\frac{d\mu}{d\lambda_1}_{|\Gamma}$ is continuous, nonnegative and
    nonzero, there is an open set~$U' \subseteq U$ such that
    $\frac{d\mu}{d\lambda_1}$ is positive on~$\Gamma \cap U'$.

Since~$\Gamma$ is~$C^2$, its curvature is a continuous
  function. As~$\mu$ does not charge lines, this curvature is nonzero
  and there exists an open subset~$U''$ of~$U$ such that $U'' \cap
  \Gamma$ is non-empty and~$\Gamma$ has positive curvature on~$U''$.
  Up to passing to a smaller neighborhood, we can find an arc~$\gamma$
  of our curve that has positive length and that is entirely on its
  convex hull. Hence, any~$k$ points on~$\gamma$ are in convex
  position. Moreover, $\frac{d\mu}{d\lambda_1}_{|\Gamma}$ is positive
  on~$U'$ and therefore on~$\gamma$, so $\mu(\gamma)$ is positive. It follows that 
\[ \forall k, \quad p(\diamond_k, \mu) \ge \pth{\frac{\mu(\gamma)}{\mu(\R^2)}}^k,\]
  which proves the second statement.
\end{proof}

\appendix

\section{Some elementary properties from measure theory}

In the proof of \cref{lem:stepfunction} we use the following property coming from measure theory.

\begin{lemma}\label{lem:basic}
  If~$(J,\mu)$ is a probability space and~$X$ is a non-empty family
  of measurable subsets of~$J$ such that 
  \begin{itemize}
  \item $X$ generates the $\sigma$-algebra of measurable sets; and
  \item $X$ is stable under finite unions and complementary operations,
  \end{itemize}
  then for every~$\varepsilon>0$ and every measurable set~$A$
  there is~$B\in X$ such that $\mu(A\triangle B)\leq\varepsilon$,
  where~$A\triangle B$ stands for the symmetric difference of~$A$ and~$B$,
  that is~$A\triangle B=(A\setminus B)\cup(B\setminus A)$.
\end{lemma}
\begin{proof}  
  One can prove the statement above by showing that the family of
  sets~$A$ for which there is indeed such an element~$B\in X$
  contains~$X$ and is stable under taking complements and countable
  unions.

  The result is true if~$A \in X$, since it then suffices to
  take~$B=A$.  Assuming that~$A$ satisfies that $\mu(A\triangle
  B)\leq\varepsilon$ for some~$B \in X$, the
  complement~$\bar{A}$ of~$A$ (in~$J$) also satisfies
  $\mu(\bar{A}\triangle\bar{B})=\mu(A\triangle B)\leq\varepsilon$
  and~$\bar{B}$ belongs to~$X$ since~$X$ is stable under taking complements.
    Let~${(A_i)}_{i\in\mathbf{N}}$ be a countable family
  of measurable subsets satisfying the property and let us prove that
  the property holds for the set $A=\bigcup_{i\in\mathbf{N}}A_i$.  For
  each~$i\in\mathbf{N}$, we know that there is a set~$B_i\in X$ such
  that~$\mu(A_i\triangle B_i) \leq \varepsilon/2^i$.  Taking
  $S_k=\bigcup_{i=1}^kB_i$ for~$k\in\mathbf{N}$ and
  $S_\infty=\bigcup_{i\in\mathbf{N}}B_i$, we have $\mu(A\triangle
  S_\infty)\leq \sum_{i\in\mathbf{N}}\mu(A_i\triangle B_i)\leq
    2\varepsilon$.  Note that ${(S_k)}_{k\in\mathbf{N}}$ is an increasing
  sequence of sets of~$X$ whose union is~$S_\infty$.  Since~$\mu$ is a
  probability measure, $\mu(S_\infty)$ is finite.  It follows that the
    real number sequence~${(\mu(S_k))}_{k\in\mathbf{N}}$ tends
  to~$\mu(S_\infty)$ as~$k$ tends to infinity.  Let~$k\in\mathbf{N}$
  be an index such that $\mu(S_\infty)-\mu(S_k) \leq\varepsilon$.
  Then $\mu(S_\infty \triangle S_k)=\mu(S_\infty)-\mu(S_k)
  \leq\varepsilon$ because $S_k\subseteq S_\infty$.  It thus follows
  that $\mu(A\triangle S_k)\leq\mu(A\triangle S_\infty)+
  \mu(S_\infty\triangle S_k)\leq 3\varepsilon$.
\end{proof}

We focus on Borel measures of~$\mathbf{R}^n$. In this case, the atoms
are precisely the singletons with positive measure.

\begin{lemma}[Sikorski~\cite{Sik58}]\label{lem-subsetofmeasurex}
  Let~$\mu$ be an atom-free measure on a set~$J$
  and let~$A\subseteq J$ be a measurable set with finite $\mu$-measure.
  Then for every every non-negative number~$x \leq \mu(A)$,
  there is a measurable subset~$B\subseteq A$ with~$\mu(B)=x$.
\end{lemma}
\begin{proof}
  If~$y$ is a real number and~$B_1,B_2\subseteq J$ are measurable sets
  satisfying~$\mu(B_1)\leq y\leq\mu(B_2)$, we define
  \[
    \alpha(B_1,B_2,y) = \sup\sst{\mu(X)}{\mu(X)\leq y, B_1\subseteq X\subseteq B_2}
  \]
  where only measurable sets~$X$ are considered.

  We first prove that for every measurable set~$B$ and real number~$y$ such that~$y\leq\mu(B)<\infty$,
  there exists a measurable set $C\subseteq B$ satisfying~$\mu(C)=\alpha(C,B,y)$.
  To see this, we fix a sequence~${(\varepsilon_n)}_{n\in\N}$ of positive numbers tending to~$0$.
  Next, we define an increasing sequence of sets~${(C_n)}_{n\in\N}$
  satisfying~$\mu(C_n)\leq y$ for every~$n\in\N$ as follows.
  Set~$C_0=\varnothing$ and observe that by the definition of~$\alpha$, for each~$i\geq 1$
  there exists a measurable set~$C_i$ such that~$C_{i-1}\subseteq C_i\subseteq B$ and~$\mu(C_i) \geq \alpha(C_{i-1},B,y)-\varepsilon_i$.
  Set~$C=\bigcup_{n\in\N}C_n$ and note that
  $\alpha(C,B,y)\leq\alpha(C_n,B,y)$ since~$C_n\subseteq C$ for every~$n\in\N$.
  It follows that
  \[\forall n\ge1,\quad
    \mu(C_n) \geq \alpha(C_{n-1},B,y)-\varepsilon_n
    \geq \alpha(C,B,y)-\varepsilon_n.
  \]
  Letting~$n$ tends to infinity yields that~$\mu(C)\geq \alpha(C,B,y)$.
  This upper bound on the supremum~$\alpha(C,B,y)$ is in particular
  reached by~$C$, so $\mu(C)=\alpha(C,B,y)$.
  This proves the property stated.

  Since~$x\le\mu(A)<\infty$, we thus know that
    there exists a measurable set~$C_1\subseteq A$ such that~$\mu(C_1)=\alpha(C_1,A,x)\leq x$.
  Further, as $\mu(A)-x\le\mu(A\setminus C_1)<\infty$, we also know that
    there exists a measurable set~$C_2\subseteq A\setminus C_1$
    such that $\mu(C_2)=\alpha(C_2,A\setminus C_1,\mu(A)-x)\leq \mu(A)-x$.
  
  As it turns out, the set~$S=A\setminus(C_1\cup C_2)$ is an atom unless it has
  measure~$0$. Indeed, suppose on the contrary that~$S$ has a measurable subset~$T$
  with~$0<\mu(T)<\mu(S)$.
  Then $\mu(C_1\cup T) > x$ since~$\alpha(C_1,A,x)=\mu(C_1)$.
  Similarly, $\mu(C_2\cup (S\setminus T)) > \mu(A)-x$ since~$\alpha(C_2,A\setminus C_1,\mu(A)-x)=\mu(C_2)$.
  Consequently, on the one hand
    $\mu(A)<\mu(C_2\cup (S\setminus T))+\mu(C_1\cup T)$ while
    on the other hand~$A$ is the disjoint union of~$(C_1\cup T)$ and~$(C_2\cup(S\setminus T))$,
    which is a contradiction.

    Therefore, since~$\mu$ is atom-free, it follows that
    $\mu(S)=0$, \emph{i.e},~$\mu(A)=\mu(C_1)+\mu(C_2)$, which implies
    that $\mu(C_1)=x$.
\end{proof}


\begin{thebibliography}{10}

\bibitem{ACF+12}
B.~M. \'Abrego, M.~Cetina, S.~Fern\'andez-Merchant, J.~Lea\~nos, and
  G.~Salazar.
\newblock On {$\leq k$}-edges, crossings, and halving lines of geometric
  drawings of {$K_n$}.
\newblock {\em Discrete Comput. Geom.}, 48(1):192--215, 2012.

\bibitem{AbFe05}
B.~M. \'Abrego and S.~Fern\'andez-Merchant.
\newblock A lower bound for the rectilinear crossing number.
\newblock {\em Graphs Combin.}, 21(3):293--300, 2005.

\bibitem{AFLS08}
B.~M. \'Abrego, S.~Fern\'andez-Merchant, J.~Lea\~nos, and G.~Salazar.
\newblock A central approach to bound the number of crossings in a generalized
  configuration.
\newblock In {\em The {IV} {L}atin-{A}merican {A}lgorithms, {G}raphs, and
  {O}ptimization {S}ymposium}, volume~30 of {\em Electron. Notes Discrete
  Math.}, pages 273--278. Elsevier Sci. B. V., Amsterdam, 2008.

\bibitem{Xing}
B.~M. \'Abrego, S.~Fern\'andez-Merchant, and G.~Salazar.
\newblock The rectilinear crossing number of {$K_n$}: closing in (or are we?).
\newblock In {\em Thirty essays on geometric graph theory}, pages 5--18.
  Springer, New York, 2013.

\bibitem{aak-eotsp-01}
O.~Aichholzer, F.~Aurenhammer, and H.~Krasser.
\newblock {{Enumerating Order Types for Small Point Sets with Applications}}.
\newblock In {\em Proc. $17^{th}$ Ann. ACM Symp. Computational Geometry}, pages
  11--18, Medford, Massachusetts, USA, 2001.

\bibitem{AGOR07}
O.~Aichholzer, J.~Garc{\'\i}a, D.~Orden, and P.~Ramos.
\newblock New lower bounds for the number of {$(\leq k)$}-edges and the
  rectilinear crossing number of {$K_n$}.
\newblock {\em Discrete Comput. Geom.}, 38(1):1--14, 2007.

\bibitem{Oti}
G.~Aloupis, J.~Iacono, S.~Langerman, O.~\"{O}zkan, and S.~Wuhrer.
\newblock The complexity of order type isomorphism.
\newblock In {\em Proceedings of the Twenty-Fifth Annual ACM-SIAM Symposium on
  Discrete Algorithms}, SODA '14, pages 405--415, 2014.

\bibitem{ARW17}
F.~Ardila, F.~Rinc\'on, and L.~Williams.
\newblock Positively oriented matroids are realizable.
\newblock {\em J. Eur. Math. Soc. (JEMS)}, 19(3):815--833, 2017.

\bibitem{hill}
J.~{Balogh}, B.~{Lidick{\'y}}, and G.~{Salazar}.
\newblock {Closing in on Hill's conjecture}.
\newblock {\em ArXiv e-prints}, Nov. 2017.

\bibitem{Borchers:1999}
B.~Borchers.
\newblock {CSDP, A C} library for semidefinite programming.
\newblock {\em Optimization Methods and Software}, 11(1-4):613--623, 1999.

\bibitem{problembook}
P.~Brass, W.~Moser, and J.~Pach.
\newblock {\em Research Problems in Discrete Geometry}.
\newblock Springer, 2005.

\bibitem{BGVV14}
S.~Brazitikos, A.~Giannopoulos, P.~Valettas, and B.-H. Vritsiou.
\newblock {\em Geometry of isotropic convex bodies}, volume 196 of {\em
  Mathematical Surveys and Monographs}.
\newblock American Mathematical Society, Providence, RI, 2014.

\bibitem{dJo17}
R.~de~Joannis~de Verclos.
\newblock {\em Applications of limits of combinatorial structures in graph
  theory and in geometry.}
\newblock PhD thesis, Universit\'e Grenoble-Alpes, 2017.

\bibitem{DHH17}
M.~G. Dobbins, A.~Holmsen, and A.~Hubard.
\newblock Realization spaces of arrangements of convex bodies.
\newblock {\em Discrete Comput. Geom.}, 58(1):1--29, 2017.

\bibitem{ERZ}
F.~Duque, R.~Fabila-Monroy, and C.~Hidalgo-Toscano.
\newblock Point sets with small integer coordinates and no large convex
  polygons.
\newblock {\em Discrete Computational Geometry}, 59, 2018.

\bibitem{Erd47}
P.~Erd\H{o}s.
\newblock Some remarks on the theory of graphs.
\newblock {\em Bull. Amer. Math. Soc.}, 53:292--294, 1947.

\bibitem{Erd84}
P.~Erd{\H o}s.
\newblock Some old and new problems in combinatorial geometry.
\newblock In {\em Convexity and graph theory ({J}erusalem, 1981)}, volume~87 of
  {\em North-Holland Math. Stud.}, pages 129--136. North-Holland, Amsterdam,
  1984.

\bibitem{ErGu73}
P.~Erd{\H o}s and R.~K. Guy.
\newblock Crossing number problems.
\newblock {\em Amer. Math. Monthly}, 80:52--58, 1973.

\bibitem{ErSz35}
P.~Erd{\H o}s and G.~Szekeres.
\newblock A combinatorial problem in geometry.
\newblock {\em Compositio Math.}, 2:463--470, 1935.

\bibitem{FaLo14}
R.~Fabila-Monroy and J.~L{\'o}pez.
\newblock Computational search of small point sets with small rectilinear
  crossing number.
\newblock {\em Journal of Graph Algorithms and Applications}, 18(3):393--399,
  2014.

\bibitem{Jan09}
S.~Janson.
\newblock Standard representation of multivariate functions on a general
  probability space.
\newblock {\em Electron. Commun. Probab.}, 14:343--346, 2009.

\bibitem{Jan13}
S.~Janson.
\newblock {\em Graphons, cut norm and distance, couplings and rearrangements},
  volume~4 of {\em New York Journal of Mathematics. NYJM Monographs}.
\newblock State University of New York, University at Albany, Albany, NY, 2013.

\bibitem{KKS70}
J.~D. Kalbfleisch, J.~G. Kalbfleisch, and R.~G. Stanton.
\newblock A combinatorial problem on convex {$n$}-gons.
\newblock In {\em Proc. {L}ouisiana {C}onf. on {C}ombinatorics, {G}raph
  {T}heory and {C}omputing ({L}ouisiana {S}tate {U}niv., {B}aton {R}ouge,
  {L}a., 1970)}, pages 180--188. Louisiana State Univ., Baton Rouge, La., 1970.

\bibitem{KaPa11}
E.~Katz and S.~Payne.
\newblock Realization spaces for tropical fans.
\newblock In {\em Combinatorial aspects of commutative algebra and algebraic
  geometry}, volume~6 of {\em Abel Symp.}, pages 73--88. Springer, Berlin,
  2011.

\bibitem{Lov12}
L.~Lov\'asz.
\newblock {\em Large networks and graph limits}, volume~60 of {\em American
  Mathematical Society Colloquium Publications}.
\newblock American Mathematical Society, Providence, RI, 2012.

\bibitem{LoSz06}
L.~Lov\'asz and B.~Szegedy.
\newblock Limits of dense graph sequences.
\newblock {\em J. Combin. Theory Ser. B}, 96(6):933--957, 2006.

\bibitem{LVWW04}
L.~Lov\'asz, K.~Vesztergombi, U.~Wagner, and E.~Welzl.
\newblock Convex quadrilaterals and {$k$}-sets.
\newblock In {\em Towards a theory of geometric graphs}, volume 342 of {\em
  Contemp. Math.}, pages 139--148. Amer. Math. Soc., Providence, RI, 2004.

\bibitem{Nik30}
O.~Nikodym.
\newblock Sur une g{\'e}n{\'e}ralisation des int{\'e}grales de {M.~J.~R}adon.
\newblock {\em Fundamenta Mathematicae}, 15:131--179, 1930.

\bibitem{Rad13}
J.~Radon.
\newblock Theorie und anwendungen der absolut additiven mengenfunktionen.
\newblock {\em Sitzungsber. Acad. Wissen. Wien}, 122:1295--1438, 1913.

\bibitem{Razborov:2007}
A.~A. Razborov.
\newblock {Flag algebras}.
\newblock {\em J. Symbolic Logic}, 72(4):1239--1282, 2007.

\bibitem{santalo1953introduction}
L.~A. Santal{\'o}.
\newblock {\em Introduction to integral geometry}, volume 1198.
\newblock Hermann, 1953.

\bibitem{ScWi94}
E.~R. Scheinerman and H.~S. Wilf.
\newblock The rectilinear crossing number of a complete graph and {S}ylvester's
  ``four point problem'' of geometric probability.
\newblock {\em Amer. Math. Monthly}, 101(10):939--943, 1994.

\bibitem{Stretchability}
P.~W. Shor.
\newblock Stretchability of pseudolines is {NP}-hard.
\newblock volume~4 of {\em Applied geometry and discrete mathematics}, pages
  531--554. Amer. Math. Soc., 1991.

\bibitem{Sik58}
R.~Sikorski.
\newblock {\em Funkcje rzeczywiste. {V}ol. {I}}.
\newblock Monografie Matematyczne. Tom XXXV. Pa\'nstwowe Wydawnictwo Naukowe,
  Warszawa, 1958.

\bibitem{sw:sage}
W.~A. Stein et~al.
\newblock {\em {S}age {M}athematics {S}oftware ({V}ersion 6.1)}.
\newblock The Sage Development Team, 2013.
\newblock \url{http://www.sagemath.org}.

\bibitem{Helly-interior}
E.~Steinitz.
\newblock Bedingt konvergente {R}eihen und konvexe {S}ysteme. i-ii-iii.
\newblock {\em J. Reine Angew. Math.}, 143 (1913), 128--175, 144 (1914), 1--40,
  146 (1916), 1-52.

\bibitem{Suk-ES}
A.~Suk.
\newblock On the {E}rd{\H o}s-{S}zekeres convex polygon problem.
\newblock {\em J. Amer. Math. Soc.}, 30:1047--1053, 2017.

\bibitem{SzPe06}
G.~Szekeres and L.~Peters.
\newblock Computer solution to the 17-point {E}rd{\H o}s-{S}zekeres problem.
\newblock {\em ANZIAM J.}, 48(2):151--164, 2006.

\bibitem{Valtr}
P.~Valtr.
\newblock Probability that $n$ random points are in convex position.
\newblock {\em Discrete \& Computational Geometry}, 13(1):637--643, 1995.

\end{thebibliography}
\end{document}